\documentclass[10pt,twoside]{article}
\usepackage[utf8]{inputenc}

\title{\uppercase{\textbf{\large{Weibel vanishing and the projective bundle formula for mixed characteristic motivic cohomology}}}}

\author{\textsc{tess bouis}}
\date{}
\usepackage[left=3cm, right=3cm, top=1in, bottom=1in, headheight=2cm]{geometry}

\usepackage{fancyhdr}
\usepackage{stmaryrd}

\usepackage{mathdots}

\usepackage{amsmath, amscd, amsfonts, amssymb, amsthm,todonotes,eurosym,enumitem,relsize}

\usepackage[utf8]{inputenc}
\usepackage{dirtytalk}
\usepackage[T1]{fontenc}
\usepackage{mathrsfs}

\usepackage[colorlinks]{hyperref}
\usepackage[figure,table]{hypcap}
\definecolor{imperialred}{RGB}{237, 41, 57}
\definecolor{royalblue}{RGB}{64, 106, 212}
\definecolor{link}{RGB}{11,0,128}
\definecolor{gren}{RGB}{32,130,63}
\hypersetup{
	bookmarksnumbered,
	pdfstartview={FitH},
	citecolor=royalblue,
	linkcolor=imperialred,
	urlcolor=link,
	linktocpage = true
	pdfpagemode={UseOutlines}
}

\usepackage{xcolor}
\usepackage{comment}

\usepackage[object=pgfhan]{pgfornament}

\usepackage{graphicx}
\usepackage{array}
\usepackage{csquotes}
\usepackage{extarrows}

\usepackage{needspace}

\input{xy}
\xyoption{all}
\usepackage{tikz-cd}
\usepackage{textcomp}
\usepackage{colonequals}

\usepackage{sectsty}
\sectionfont{\sc\centering}
\subsectionfont{\bf\large}
\subsubsectionfont{\bf\normalsize}

\usepackage{fancyhdr}

\newlength{\outermargin} 
\newlength{\mar} 
\newlength{\len}
\newlength{\temp}
\newtheorem{theorem}{Theorem}[section]
\newtheorem{lemma}[theorem]{Lemma}
\newtheorem{corollary}[theorem]{Corollary}
\newtheorem{proposition}[theorem]{Proposition}

\newtheorem{theoremintro}{Theorem}

\newtheorem{corollaryintro}[theoremintro]{Corollary}

\theoremstyle{definition}
\newtheorem{remark}[theorem]{Remark}

\newtheorem{example}[theorem]{Example}

\newtheorem{construction}[theorem]{Construction}
\newtheorem{definition}[theorem]{Definition}
\newtheorem{notation}[theorem]{Notation}

\newtheorem{conjecture}[theorem]{Conjecture}

\DeclareMathOperator{\Z}{\mathbb{Z}}
\DeclareMathOperator{\Q}{\mathbb{Q}}

\DeclareMathOperator{\F}{\mathbb{F}}

\renewcommand{\epsilon}{\varepsilon}

\DeclareFontFamily{U}{MnSymbolC}{}
\DeclareFontShape{U}{MnSymbolC}{m}{n}{
	<-5.5> MnSymbolC5
	<5.5-6.5> MnSymbolC6
	<6.5-7.5> MnSymbolC7
	<7.5-8.5> MnSymbolC8
	<8.5-9.5> MnSymbolC9
	<9.5-11.5> MnSymbolC10
	<11.5-> MnSymbolCb12
}{}

\usepackage[bbgreekl]{mathbbol}
\DeclareSymbolFontAlphabet{\mathbb}{AMSb}
\DeclareSymbolFontAlphabet{\mathbbl}{bbold}
\newcommand{\Prism}{{\mathlarger{\mathbbl{\Delta}}}}

\usepackage{soul}
\usepackage{enumitem}
\numberwithin{equation}{theorem}

\begin{document}
	
	\maketitle
	
	\pagestyle{fancy}
	\fancyhead[EC]{TESS BOUIS}
	\fancyhead[OC]{\uppercase{Weibel vanishing and the projective bundle formula}}
	\fancyfoot[C]{\thepage}
	
	\begin{abstract}
		We prove that the motivic cohomology of mixed characteristic schemes, introduced in our previous work, satisfies various expected properties of motivic cohomology, including a motivic refinement of Weibel's vanishing in algebraic $K$-theory, the projective bundle formula, a comparison to Milnor $K$-theory, and a universal characterisation in terms of pro cdh descent. 
        These results extend those of Elmanto--Morrow to schemes which are not necessarily defined over a field.
	\end{abstract}

	{
		\hypersetup{linkcolor=black}
		\tableofcontents
	}
	
	\section{Introduction}
	
	\vspace{-\parindent}
	\hspace{\parindent}

    This paper is a sequel to \cite{bouis_motivic_2024}, where we introduced a theory of motivic complexes $\Z(i)^{\text{mot}}(X)$ for general quasi-compact quasi-separated (qcqs) schemes $X$. This theory extends the construction of Elmanto--Morrow \cite{elmanto_motivic_2023} to schemes which are not necessarily defined over a field, and recovers the classical theory of motivic cohomology when $X$ is smooth over a field \cite{elmanto_motivic_2023} or over a Dedekind domain \cite{bouis_beilinson-lichtenbaum_2025}. Our goal here is to prove several structural properties of the motivic complexes $\Z(i)^{\text{mot}}(X)$ for these possibly singular, mixed characteristic schemes $X$.
    
    
    
    One of the most interesting, yet mysterious features of the algebraic $K$-theory of singular schemes is the presence of nonzero negative $K$-groups. Most of the current understanding of negative $K$-groups relies on results on the behaviour of algebraic $K$-theory with respect to blowups \cite{cortinas_cyclic_2008,kerz_algebraic_2018}. It was proved in particular by Thomason \cite{thomason_K-groupes_1993} that algebraic $K$-theory sends the blowup square associated to a regular closed immersion to a long exact sequence of $K$-groups. The following result is a cohomological refinement of Thomason's result.
	
	\begin{theoremintro}[Regular blowup formula; see Theorem~\ref{theoremregularblowupformula}]\label{theoremintroregularblowupexcision}
		For every regular closed immersion $Y \rightarrow X$ of qcqs schemes ({\it i.e.}, the closed subscheme $Y$ is Zariski-locally on $X$ defined by a regular sequence) and every integer $i \geq 0$, the commutative diagram
		$$\begin{tikzcd}
			\Z(i)^{\emph{mot}}(X) \ar[r] \ar[d] & \Z(i)^{\emph{mot}}(Y) \ar[d] \\
			\Z(i)^{\emph{mot}}(\emph{Bl}_Y(X)) \ar[r] & \Z(i)^{\emph{mot}}(\emph{Bl}_Y(X) \times_X Y)
		\end{tikzcd}$$
		is a cartesian square in the derived category $\mathcal{D}(\Z)$.
	\end{theoremintro}
	
	However, algebraic $K$-theory fails to associate long exact sequences to general blowups. Motivated by Grothendieck's theorem on formal functions for quasi-coherent cohomology \cite[théorème~$4.1.5$]{grothendieck_elements_1961}, many people hoped for a formal analogue of these long exact sequences in algebraic $K$-theory, that would hold for general blowups. This was finally proved by Kerz--Strunk--Tamme \cite{kerz_algebraic_2018}, in the form of a pro cdh excision property for the algebraic $K$-theory of arbitrary noetherian schemes. Note that the pro cdh topology was recently introduced by Kelly--Saito \cite{kelly_procdh_2024}, as a way to encode this pro cdh excision property in a descent property for this Grothendieck topology. The following result is one of our main technical results, and relies in particular on the description of the motivic complexes $\Z(i)^{\text{mot}}$ in terms of quasi-coherent information (\cite[Theorem~B]{bouis_motivic_2024}) and on Grothendieck's theorem on formal functions, thus shedding some light on the analogy between quasi-coherent and $K$-theoretic techniques.
	
	\begin{theoremintro}[Pro cdh descent; see Theorem~\ref{theoremprocdhdescentformotiviccohomology}]\label{theoremintroprocdhdescent}
		For every integer $i \geq 0$, the presheaf $\Z(i)^{\emph{mot}}$ satisfies pro cdh descent on noetherian schemes. That is, for every abstract blowup square
		$$\begin{tikzcd}
			Y' \ar[r] \ar[d] & X' \ar[d] \\
			Y \ar[r] & X
		\end{tikzcd}$$
		of noetherian schemes, the associated commutative diagram
		$$\begin{tikzcd}
			\Z(i)^{\emph{mot}}(X) \ar[r] \ar[d] & \Z(i)^{\emph{mot}}(X') \ar[d] \\
			\{\Z(i)^{\emph{mot}}(rY)\}_r \ar[r] & \{\Z(i)^{\emph{mot}}(rY')\}_r
		\end{tikzcd}$$
		is a weakly cartesian square of pro objects in the derived category $\mathcal{D}(\Z)$.
	\end{theoremintro}

    Theorem~\ref{theoremintroprocdhdescent} is formally equivalent to the analogous statement for the presheaves $\Z(i)^{\text{TC}}$, introduced in \cite{bouis_motivic_2024}. Over a field, this was proved by Elmanto--Morrow by reduction to powers of the cotangent complex \cite[Section~$8.1$]{elmanto_motivic_2023}. In mixed characteristic, the main new difficulty is to prove pro cdh descent for the profinite completion of $\Z(i)^{\text{TC}}$, which is naturally expressed as a product over all prime numbers $p$ of contributions coming from integral $p$-adic Hodge theory. Surprisingly, this fact ultimately relies on a statement purely in terms of the derived de Rham cohomology of characteristic zero schemes (see Section~\ref{subsectionprocdhdescentmotiviccohomology}). In particular, our proof of Theorem~\ref{theoremintroprocdhdescent} does not follow the same strategy as Kerz--Strunk--Tamme's proof of pro cdh descent for algebraic $K$-theory (see Remark~\ref{remarkprocdhdescentforKtheory}).
	
	Kelly--Saito moreover proved that non-connective algebraic $K$-theory not only satisfies pro cdh descent, but is the pro cdh sheafification of connective algebraic $K$-theory \cite{kelly_procdh_2024}. Combined with the observation of Bhatt--Lurie that connective algebraic $K$-theory is left Kan extended on commutative rings from smooth $\Z$-algebras \cite{elmanto_modules_2020}, this motivated the definition of the {\it pro cdh motivic complexes} $\Z(i)^{\text{procdh}}$, as the pro cdh sheafification of the left Kan extension of the classical motivic complexes~$\Z(i)^{\text{cla}}$ given by Bloch's cycle complexes.
	
	\begin{corollaryintro}[Comparison to pro cdh motivic cohomology; see Theorem~\ref{theoremcomparisonprocdhmotivic}]\label{corollaryintrocomparisonprocdhmotivic}
		Let $X$ be a noetherian scheme. Then for every integer $i \geq 0$, there is a natural equivalence
		$$\Z(i)^{\emph{procdh}}(X) \xlongrightarrow{\sim} \Z(i)^{\emph{mot}}(X)$$
		in the derived category $\mathcal{D}(\Z)$.
	\end{corollaryintro}

    The proof of Corollary~\ref{corollaryintrocomparisonprocdhmotivic} relies on pro cdh descent (Theorem~\ref{theoremintroprocdhdescent}) and on the following comparison to lisse motivic cohomology. 
    
    Note that the pro cdh motivic complexes $\Z(i)^{\text{procdh}}$ are not finitary, so they cannot coincide with the motivic complexes $\Z(i)^{\text{mot}}$ on general qcqs schemes. However, together with the finitariness of the Nisnevich sheaves $\Z(i)^{\text{mot}}$ proved in \cite[Theorem~C]{bouis_motivic_2024}, this gives a universal cycle-theoretic characterisation of the motivic complexes $\Z(i)^{\text{mot}}$ on general qcqs schemes.

    \begin{theoremintro}[Comparison to lisse motivic cohomology; see Corollary~\ref{corollarylissemotivicmaincomparisontheorem}]\label{theoremintrolissemotiviccohomology}
        Let $A$ be a local ring. Then for every integer $i \geq 0$, there is a natural equivalence
        $$\Z(i)^{\emph{lisse}}(A) \xlongrightarrow{\sim} \tau^{\leq i} \Z(i)^{\text{mot}}(A)$$
        in the derived category $\mathcal{D}(\Z)$, where $\Z(i)^{\emph{lisse}}$ denotes the weight-$i$ lisse motivic cohomology of~$A$, defined as the left Kan extension from smooth $\Z$-algebras of the classical motivic complex $\Z(i)^{\emph{cla}}$. In particular, the functor $\tau^{\leq i} \Z(i)^{\emph{mot}}$ is left Kan extended on local rings from local essentially smooth $\Z$-algebras.
    \end{theoremintro}

    Theorem~\ref{theoremintrolissemotiviccohomology} was first proved by Elmanto--Morrow for local rings over a field \cite[Theorem~$7.7$]{elmanto_motivic_2023}, by using the projective bundle formula and the comparison to classical motivic cohomology. We provide a more direct argument, which does not rely on these ingredients, for this comparison to lisse motivic cohomology, thus reproving in particular the result of Elmanto--Morrow in equicharacteristic.

    For $i=1$, this comparison to lisse motivic cohomology also implies that the motivic complex $\Z(1)^{\text{mot}}$ is related in the expected way to the unit and Picard groups (see Example~\ref{exampleweightonemotiviccohomology}). In turn, this relation to the unit group, together with the multiplicative structure on the motivic complexes $\Z(i)^{\text{mot}}$, induces symbol maps from Milnor $K$-theory to the Milnor range of motivic cohomology, and we prove the following.

    \begin{theoremintro}[Milnor $K$-theory; see Theorem~\ref{theoremcomparisontoMilnorKtheory})]\label{theoremintroMilnorKtheory}
        Let $A$ be a henselian local ring. Then for any integers $i \geq 0$ and $n \geq 1$, there is a natural isomorphism
			$$\widehat{\emph{K}}{}^{\emph{M}}_i(A)/n \xlongrightarrow{\cong} \emph{H}^i_{\emph{mot}}(A,\Z(i))/n$$
			of abelian groups, where $\widehat{\emph{K}}{}^{\emph{M}}_i(A)$ denotes the $i^{\emph{th}}$ improved Milnor $K$-group of $A$ in the sense of \cite{kerz_milnor_2010}.
    \end{theoremintro}

    The proof of Theorem~\ref{theoremintroMilnorKtheory} relies in particular on results of Lüders--Morrow in $p$-adic Hodge theory \cite{luders_milnor_2023}. Note that for local rings over a field, Elmanto--Morrow proved this result with $\Z$-coefficients and without assuming the local ring to be henselian \cite[Theorem~$7.12$]{elmanto_motivic_2023}. Although we expect the same result to hold for general local rings, even the smooth case is open in mixed characteristic (see Remark~\ref{remarksmoothcaseimpliesgeneralcaseMilnorcomparison}).
    
	
    Similarly, one can use the relation to the Picard group in weight one and the multiplicative structure on the motivic complexes $\Z(i)^{\text{mot}}$ to formulate the projective bundle formula, as we explain now.

    Assuming the existence of a well-behaved derived category of motives $\mathcal{D}_{\text{mot}}(X)$, the motivic cohomology groups of a scheme $X$ should be given by
	$$\text{H}^j_{\text{mot}}(X,\Z(i)) \cong \text{Hom}_{\mathcal{D}_{\text{mot}}(X)}(\text{M}(X),\Z(i)[j]),$$
	where $\text{M}(X) \in \mathcal{D}_{\text{mot}}(X)$ is the motive associated to $X$, and $\Z(i) \in \mathcal{D}_{\text{mot}}(X)$ are the Tate motives, fitting for every integer $r \geq 0$ in a natural decomposition in $\mathcal{D}_{\text{mot}}(X)$:
	$$\text{M}(\mathbb{P}^r_{\Z}) \cong \bigoplus_{j=0}^r \Z(j)[2j].$$
	In the $\mathbb{A}^1$-invariant framework, Voevodsky constructed such a derived category of motives, in which the classical motivic complexes $\Z(i)^{\text{cla}}$ can be interpreted in terms of these Tate motives. Without assuming $\mathbb{A}^1$-invariance, Annala--Iwasa \cite{annala_motivic_2023} and Annala--Hoyois--Iwasa \cite{annala_algebraic_2025,annala_atiyah_2024} recently introduced a more general derived category of motives, where the decomposition of the motive $\text{M}(\mathbb{P}^r_{\Z})$, {\it i.e.}, the {\it projective bundle formula}, is isolated as the key defining property. The following result states that the motivic complexes $\Z(i)^{\text{mot}}$ fit within this theory of non-$\mathbb{A}^1$-invariant motives.

    \needspace{5\baselineskip}
	
	\begin{theoremintro}[Projective bundle formula; see Theorem~\ref{theoremprojectivebundleformula}]\label{theoremintroprojectivebundleformula}
		Let $X$ be a qcqs scheme, $i \geq 0$ be an integer, and $\mathcal{E}$ be a vector bundle of rank $r$ on $X$. Then for every integer $i \geq 0$, the powers of the motivic first Chern class $c_1^{\emph{mot}}(\mathcal{O}(1)) \in \emph{Pic}(\mathbb{P}_X(\mathcal{E})) \cong \emph{H}^2_{\emph{mot}}(\mathbb{P}_X(\mathcal{E}),\Z(1))$ induce a natural equivalence
		$$\bigoplus_{i=0}^{r-1} \Z(i-j)^{\emph{mot}}(X)[-2j] \xlongrightarrow{\sim} \Z(i)^{\emph{mot}}(\mathbb{P}_X(\mathcal{E}))$$
		in the derived category $\mathcal{D}(\Z)$.
	\end{theoremintro}
	
	Theorem~\ref{theoremintroprojectivebundleformula} is proved by Elmanto--Morrow in the equicharacteristic case \cite{elmanto_motivic_2023}, where the proof relies on the projective bundle formula for the complexes $\Z(i)^{\text{cdh}}$ \cite{bachmann_A^1-invariant_2024}. In mixed characteristic, however, the cdh-local motivic complexes $\Z(i)^{\text{cdh}}$ are known to satisfy the projective bundle formula only conditionally on a certain property of valuation rings, called $F$-smoothness \cite{bhatt_syntomic_2023,bachmann_A^1-invariant_2024}. This condition can be proved in mixed characteristic for valuation rings over a perfectoid base: this is the main result of \cite{bouis_cartier_2023}. The case of general valuation rings remaining open, our proof of Theorem~\ref{theoremintroprojectivebundleformula} is different from that of Elmanto--Morrow, and uses in particular our description of motivic cohomology with finite coefficients in terms of syntomic cohomology (\cite[Theorem~$5.10$]{bouis_motivic_2024}) and a special case of Theorem~\ref{theoremintroregularblowupexcision}.

    Finally, we explain how the previous results, and in particular pro cdh descent, can be used to give a motivic description of the negative $K$-groups of singular schemes.
	
	An important conjecture of Weibel \cite{weibel_K-theory_1980} states that for every noetherian scheme $X$ of dimension at most $d$, the negative $K$-groups $\text{K}_{-n}(X)$ vanish for integers $n > d$. This conjecture was settled by Kerz--Strunk--Tamme \cite{kerz_algebraic_2018}, as a consequence of pro cdh descent for algebraic $K$-theory. The proof of the following result uses the techniques of Kerz--Strunk--Tamme \cite{kerz_algebraic_2018} as reformulated by Elmanto--Morrow \cite{elmanto_motivic_2023}, who proved the same result over a field. In particular, Theorem~\ref{theoremintroweibelvanishing} relies on pro cdh descent for motivic cohomology (Theorem~\ref{theoremintroprocdhdescent}) and on Antieau--Mathew--Morrow--Nikolaus' rigidity theorem for syntomic cohomology (\cite[Theorem~$5.2$]{antieau_beilinson_2020}).
	
	\begin{theoremintro}[Motivic Weibel vanishing; see Theorem~\ref{theoremmotivicWeibelvanishing}]\label{theoremintroweibelvanishing}
		Let $X$ be a noetherian scheme of finite dimension $d$, and $i \geq 0$ be an integer. Then for every integer $j > i+d$, the motivic cohomology group $\emph{H}^j_{\emph{mot}}(X,\Z(i))$ is zero.
	\end{theoremintro}

	Via the Atiyah--Hirzebruch spectral sequence relating motivic cohomology to algebraic $K$-theory (\cite[Theorem~C]{bouis_motivic_2024}), Theorem~\ref{theoremintroweibelvanishing} is a motivic refinement of Weibel's vanishing conjecture in $K$\nobreakdash-theory. Moreover, this result also implies a new cohomological description of the lowest nonzero negative $K$\nobreakdash-groups (see Corollary~\ref{corollaryweibelupgraded}).

	\subsection*{Notation}
	
	\vspace{-\parindent}
	\hspace{\parindent}



    Given a commutative ring $R$, an $R$-algebra $S$, and a scheme $X$ over $\text{Spec}(R)$, denote by~$X_{S}$ the base change $X \times_{\text{Spec}(R)} \text{Spec}(S)$ of $X$ from $R$ to $S$. If $X$ is a derived scheme, this base change is implicitly the derived base change from $R$ to $S$. We sometimes use the derived base even on classical schemes, and say explicitly when we do so.

    Given a commutative ring $R$, denote by $\mathcal{D}(R)$ the derived category of $R$-modules; it is implicitly the derived $\infty$-category of $R$-modules, and is in particular naturally identified with the category of $R$-linear spectra. Given an element $d$ of $R$, also denote by $(-)^\wedge_d$ the $d$-adic completion functor in the derived category~$\mathcal{D}(R)$.

    Given a commutative ring $R$ and an ideal $I$ of $R$, the pair $(R,I)$ is called {\it henselian} if it satisfies Hensel's lemma. A local ring $R$ is called {\it henselian} if the pair $(R,\mathfrak{m})$ is henselian, where $\mathfrak{m}$ is the maximal ideal of $R$. Henselian local rings are the local rings for the Nisnevich topology. A commutative ring $R$ is called {\it $d$-henselian}, for $d$ an element of $R$, if the pair $(R,(d))$ is henselian. A functor $F(-)$ on commutative rings is called {\it rigid} if for every henselian pair $(R,I)$, the natural map $F(R) \rightarrow F(R/I)$ is an equivalence.

    We use several Grothendieck topologies, including the Zariski, Nisnevich, and cdh topologies. Denote by $L_{\text{Zar}}$, $L_{\text{Nis}}$, and $L_{\text{cdh}}$ the sheafification functors for these topologies.

	\subsection*{Acknowledgements}
	
	\vspace{-\parindent}
	\hspace{\parindent}

    I am very grateful to Matthew Morrow for sharing many insights on motivic cohomology, and for careful readings of this manuscript.	I would also like to thank Jacob Lurie and Georg Tamme for many helpful comments and corrections on a preliminary version of this paper, and Ben Antieau, Denis\nobreakdash-Charles Cisinski, Dustin Clausen, Frédéric Déglise, Elden Elmanto, Quentin Gazda, Marc \hbox{Hoyois}, Ryomei Iwasa, Shane Kelly, Niklas Kipp, Arnab Kundu, Shuji Saito, and Georg Tamme for helpful discussions. This project has received funding from the European Research Council (ERC) under the European Union’s Horizon 2020 research and innovation programme (grant agreement No. 101001474).

	
	\needspace{4\baselineskip}

    \section{Comparison to Milnor \texorpdfstring{$K$}{TEXT}-theory and lisse motivic cohomology}\label{sectionmilnorandlisse}
	
	\vspace{-\parindent}
	\hspace{\parindent}

    \vspace{-\parindent}
	\hspace{\parindent}
	
	In this section, we study the motivic complexes $\Z(i)^{\text{mot}}$ on local rings. We prove that these are left Kan extended, in degrees at most $i$, from local essentially smooth $\Z$-algebras (Theorem~\ref{theoremmotiviccohomologyisleftKanextendedonlocalringsinsmalldegrees}). Via the comparison to syntomic cohomology in the smooth case (\cite[Theorem~D\,(3)]{bouis_motivic_2024}), this means that the motivic complexes~$\Z(i)^{\text{mot}}$ are controlled, up to degree~$i$, by classical motivic cohomology (Corollary~\ref{corollarylissemotivicmaincomparisontheorem}). We use the latter result to construct a comparison map from the $i^{\text{th}}$ (improved) Milnor $K$-group of a general local ring $A$ to the motivic cohomology group $\text{H}^i_{\text{mot}}(A,\Z(i))$, and prove that this map is an isomorphism with finite coefficients (Theorem~\ref{theoremcomparisontoMilnorKtheory}).
	
	\subsection{Comparison to lisse motivic cohomology}
	
	\vspace{-\parindent}
	\hspace{\parindent}
	
	In this subsection, we prove a comparison between motivic cohomology and lisse motivic cohomology on general local rings (Corollary~\ref{corollarylissemotivicmaincomparisontheorem}), which generalises to mixed characteristic the analogous comparison result of Elmanto--Morrow over a field (\cite[Theorem~$7.7$]{elmanto_motivic_2023}). To do so, we use the following comparison map, where lisse motivic cohomology is defined on animated commutative rings as the left Kan extension of classical motivic cohomology from smooth $\Z$-algebras (\cite[Definition~$3.7$]{bouis_motivic_2024}).
	
	\begin{definition}[Lisse-motivic comparison map]\label{definitionlissemotiviccomparisonmap}
		For every integer $i \in \Z$, the {\it lisse-motivic comparison map} is the map
		$$\Z(i)^{\text{lisse}}(-) \longrightarrow \Z(i)^{\text{mot}}(-)$$
		of functors from animated commutative rings to the derived category $\mathcal{D}(\Z)$ defined as the composite
		$$\big(L_{\text{AniRings}/\text{Sm}_{\Z}} \Z(i)^{\text{cla}}\big)(-) \longrightarrow \big(L_{\text{AniRings}/\text{Sm}_{\Z}} \Z(i)^{\text{mot}}\big)(-) \longrightarrow \Z(i)^{\text{mot}}(-),$$
		where the first map is the map induced by \cite[Definition~$3.23$]{bouis_motivic_2024} and the second map is the canonical map.
	\end{definition}
	
	\begin{lemma}\label{lemmarationalmotiviccohoisLKEindegreesatmost2i}
		For every integer $i \geq 0$, the functor
		$$\tau^{\leq 2i} \Q(i)^{\emph{mot}}(-) : \emph{AniRings} \longrightarrow \mathcal{D}(\Q)$$
		is left Kan extended from smooth $\Z$-algebras.
	\end{lemma}
	
	\begin{proof}
		By \cite[Example~$1.0.6$]{elmanto_modules_2020}, connective algebraic $K$-theory
		$$\tau^{\leq 0} \text{K}(-;\Q) : \text{AniRings} \longrightarrow \mathcal{D}(\Q)$$
		is left Kan extended from smooth $\Z$-algebras. By \cite[Corollary~$4.56$]{bouis_motivic_2024}, this implies that the functor
		$$\bigoplus_{i \geq 0} \tau^{\leq 0}\big(\Q(i)^{\text{mot}}(-)[2i]\big) : \text{AniRings} \longrightarrow \mathcal{D}(\Q)$$
		is left Kan extended from smooth $\Z$-algebras, which is equivalent to the desired result.
	\end{proof}
	
	\begin{corollary}\label{corollaryrationallissemotiviccomparison}
		Let $R$ be an animated commutative ring. Then for every integer $i \geq 0$, the lisse-motivic comparison map induces a natural equivalence
		$$\Q(i)^{\emph{lisse}}(R) \xlongrightarrow{\sim} \tau^{\leq 2i} \Q(i)^{\emph{mot}}(R)$$
		in the derived category $\mathcal{D}(\Q)$.
	\end{corollary}
	
	\begin{proof}
		If $R$ is a smooth $\Z$-algebra, the result is a consequence of the rational splitting of algebraic $K$\nobreakdash-theory induced by Adams operations and of the fact that, by construction, the classical motivic complex $\Z(i)^{\text{cla}}(R) \in \mathcal{D}(\Z)$ is in degrees at most~$2i$. In general, this is then a consequence of Lemma~\ref{lemmarationalmotiviccohoisLKEindegreesatmost2i}.
	\end{proof}
	
	\begin{proposition}\label{propositionrationalcomparisonlissemotivicforlocalrings}
		Let $R$ be a local ring. Then for every integer $i \geq 0$, the lisse-motivic comparison map induces a natural equivalence
		$$\Q(i)^{\emph{lisse}}(R) \longrightarrow \tau^{\leq i} \Q(i)^{\emph{mot}}(R)$$
		in the derived category $\mathcal{D}(\Q)$. Moreover, the motivic cohomology group $\emph{H}^j_{\emph{mot}}(R,\Q(i))$ is zero for \hbox{$i < j \leq 2i$}.
	\end{proposition}
	
	\begin{proof}
		The classical motivic complex $\Z(i)^{\text{cla}}(-)$ is Zariski-locally in degrees at most $i$ (\cite[Corollary~$4.4$]{geisser_motivic_2004}). By taking left Kan extension, this implies that the lisse motivic complex $\Z(i)^{\text{lisse}}(-)$ is also Zariski-locally in degrees at most $i$. In particular, the lisse motivic complex $\Q(i)^{\text{lisse}}(R)$ is in degrees at most $i$. The result is then a consequence of Corollary~\ref{corollaryrationallissemotiviccomparison}.
	\end{proof}

	\begin{remark}\label{remarkDrinfeldtheoremK-1onhenselianlocalrings}
		By Drinfeld's theorem (\cite[Theorem~$3.7$]{drinfeld_infinite_2006}), the $K$-group $\text{K}_{-1}(R)$ vanishes for every henselian local ring $R$. By \cite[Corollary~$4.60$]{bouis_motivic_2024}, this implies that for every integer $i \geq 0$, the motivic cohomology group $\text{H}^{2i+1}_{\text{mot}}(R,\Q(i))$ is zero, {\it i.e.}, that the motivic cohomology group $\text{H}^{2i+1}_{\text{mot}}(R,\Z(i))$ is torsion.
	\end{remark}
	
	\begin{corollary}\label{corollaryrationalmotiviccohomologyisleftKanextendedinsmalldegrees}
		For every integer $i \geq 0$, the functor $\tau^{\leq i} \Q(i)^{\emph{mot}}$, from local rings to the derived category $\mathcal{D}(\Q)$, is left Kan extended from local essentially smooth $\Z$-algebras.
	\end{corollary}
	
	\begin{proof}
		This is a consequence of Lemma~\ref{lemmarationalmotiviccohoisLKEindegreesatmost2i} and Proposition~\ref{propositionrationalcomparisonlissemotivicforlocalrings}.
	\end{proof}
	
	\begin{proposition}\label{propositionmodpmotiviccohoisleftKanextendedinsmalldegrees}
		Let $p$ be a prime number, and $k$ be an integer. Then for every integer~\hbox{$i \geq 0$}, the functor
		$$\tau^{\leq i} \Z/p^k(i)^{\emph{mot}}(-) : \emph{Rings} \longrightarrow \mathcal{D}(\Z/p^k)$$
		is left Kan extended from smooth $\Z$-algebras.
	\end{proposition}
	
	\begin{proof}
		By \cite[Theorem~$5.10$]{bouis_motivic_2024}, this is equivalent to the fact that the functor $\tau^{\leq i} \Z/p^k(i)^{\text{syn}}(-)$ on commutative rings is left Kan extended from smooth $\Z$-algebras. The functor $\Z/p^k(i)^{\text{syn}}(-)$ is left Kan extended from smooth $\Z$-algebras (\cite[Notation~$5.7$]{bouis_motivic_2024}), so this is equivalent to the fact that the functor $\tau^{>i} \Z/p^k(i)^{\text{syn}}(-)$ on commutative rings is left Kan extended from smooth $\Z$-algebras. By \cite[Lemma~$7.6$]{elmanto_motivic_2023}, it then suffices to prove that the functor $\tau^{>i} \Z/p^k(i)^{\text{syn}}(-)$ is rigid. To prove this, consider the fibre sequence of $\mathcal{D}(\Z/p^k)$-valued functors
		$$R\Gamma_{\text{ét}}(-,j_! \mu_{p^k}^{\otimes i}) \longrightarrow \Z/p^k(i)^{\text{syn}}(-) \longrightarrow \Z/p^k(i)^{\text{BMS}}(-)$$
		on commutative rings (\cite[Remark~$8.4.4$]{bhatt_absolute_2022}). By rigidity for étale cohomology (\cite{gabber_affine_1994}, see also \cite[Corollary~$1.18\,(1)$]{bhatt_arc-topology_2021}), the first term of this fibre sequence is rigid. The desired result is then a consequence of \cite[Theorem~$2.27$]{bouis_motivic_2024}.
	\end{proof}
	
	\begin{corollary}\label{corollarymodpmotiviccohomologyisleftKanextendedinsmalldegreesonlocalrings}
		Let $p$ be a prime number, and $k \geq 1$ be an integer. Then for every integer~\hbox{$i \geq 0$}, the functor $\tau^{\leq i} \Z/p^k(i)^{\emph{mot}}$, from local rings to the derived category $\mathcal{D}(\Z/p^k)$, is left Kan extended from local essentially smooth $\Z$-algebras.
	\end{corollary}
	
	\begin{proof}
		This is a consequence of Proposition~\ref{propositionmodpmotiviccohoisleftKanextendedinsmalldegrees}.
	\end{proof}
	
	\begin{lemma}\label{lemmatopdegreeintegraltomodpissurjective}
		Let $R$ be a local ring, $p$ be a prime number, and $k \geq 1$ be an integer. Then for every integer $i \geq 0$, the natural map of abelian groups
		$$\emph{H}^i_{\emph{mot}}(R,\Z(i)) \longrightarrow \emph{H}^i_{\emph{mot}}(R,\Z/p^k(i))$$
		is surjective.
	\end{lemma}
	
	\begin{proof}
		Let $P \rightarrow R$ be a henselian surjection, where $P$ is a local ind-smooth $\Z$-algebra. By Corollary~\ref{corollarymodpmotiviccohomologyisleftKanextendedinsmalldegreesonlocalrings}, the functor $\tau^{\leq i} \Z/p^k(i)^{\text{mot}}$ is left Kan extended on local rings from local essentially smooth $\Z$-algebras, so the natural map of abelian groups
		$$\text{H}^i_{\text{mot}}(P,\Z/p^k(i)) \longrightarrow \text{H}^i_{\text{mot}}(R,\Z/p^k(i))$$
		is surjective. That is, the right vertical map in the commutative diagram of abelian groups
		$$\begin{tikzcd}
			\text{H}^i_{\text{mot}}(P,\Z(i)) \ar[r] \ar[d] & \text{H}^i_{\text{mot}}(P,\Z/p^k(i)) \ar[d] \\
			\text{H}^i_{\text{mot}}(R,\Z(i)) \ar[r] & \text{H}^i_{\text{mot}}(R,\Z/p^k(i))
		\end{tikzcd}$$
		is surjective. To prove that the bottom horizontal map is surjective, it thus suffices to prove that the top vertical map is surjective. The local ring $P$ is a filtered colimit of local essentially smooth $\Z$-algebras, so it suffices to prove that this top vertical map is surjective for local essentially smooth $\Z$-algebras. To prove this, it suffices to prove that the motivic complex $\Z(i)^{\text{mot}}(-)$ is zero in degree~$i+1$ on local essentially smooth $\Z$-algebras, which is a formal consequence of \cite[Theorem~$5.10$]{bouis_motivic_2024} and \cite[Corollary~$4.4$]{geisser_motivic_2004}. 
	\end{proof}
	
	\begin{corollary}\label{corollaryHilbert90}
		Let $R$ be a local ring. Then for every integer $i \geq 1$, the motivic cohomology group $\emph{H}^{i+1}_{\emph{mot}}(R,\Z(i))$ is zero. If the local ring $R$ is moreover henselian, then the motivic cohomology group $\emph{H}^1_{\emph{mot}}(R,\Z(0))$ is zero. 
	\end{corollary}
	
	\begin{proof}
		By Lemma~\ref{lemmatopdegreeintegraltomodpissurjective} and the short exact sequence of abelian groups
		$$0 \longrightarrow \text{H}^i_{\text{mot}}(R,\Z(i))/p \longrightarrow \text{H}^i_{\text{mot}}(R,\F_p(i)) \longrightarrow \text{H}^{i+1}_{\text{mot}}(R,\Z(i))[p] \longrightarrow 0$$
		for every prime number $p$ and every integer $i \geq 0$, the abelian group $\text{H}^{i+1}_{\text{mot}}(R,\Z(i))$ is torsionfree. By Proposition~\ref{propositionrationalcomparisonlissemotivicforlocalrings} if $i \geq 1$, and by Remark~\ref{remarkDrinfeldtheoremK-1onhenselianlocalrings} if $i=0$ and $R$ is henselian, it is also torsion, so it is zero.
	\end{proof}
	
	\begin{theorem}\label{theoremmotiviccohomologyisleftKanextendedonlocalringsinsmalldegrees}
		For every integer $i \geq 0$, the functor $\tau^{\leq i} \Z(i)^{\emph{mot}}$, from local rings to the derived category $\mathcal{D}(\Z)$, is left Kan extended from local essentially smooth $\Z$-algebras.
	\end{theorem}
	
	\begin{proof}
		It suffices to prove the result rationally, and modulo $p$ for every prime number $p$. The result rationally is Corollary~\ref{corollaryrationalmotiviccohomologyisleftKanextendedinsmalldegrees}. Let $p$ be a prime number. For every local ring $R$, the natural map of abelian groups
		$$\text{H}^i_{\text{mot}}(R,\Z(i)) \longrightarrow \text{H}^i_{\text{mot}}(R,\F_p(i))$$
		is surjective by Lemma~\ref{lemmatopdegreeintegraltomodpissurjective}, so the natural map
		$$\big(\tau^{\leq i} \Z(i)^{\text{mot}}(R)\big)/p \longrightarrow \tau^{\leq i} \F_p(i)^{\text{mot}}(R)$$
		is an equivalence in the derived category $\mathcal{D}(\F_p)$. The result modulo $p$ is then Corollary~\ref{corollarymodpmotiviccohomologyisleftKanextendedinsmalldegreesonlocalrings}.
	\end{proof}
	
	Note that the proof of Theorem~\ref{theoremmotiviccohomologyisleftKanextendedonlocalringsinsmalldegrees} is similar to the proof of Elmanto--Morrow in equicharacteristic. The following consequence, however, uses the comparison to syntomic cohomology \cite[Theorem~$5.10$]{bouis_motivic_2024} in the case of smooth $\Z$-schemes. The proof of the latter, in any characteristic, is somehow simpler than the proof of Elmanto--Morrow's (stronger) comparison result to classical motivic cohomology: in particular, it does not use a presentation lemma, or the projective bundle formula. The proof of Corollary~\ref{corollarylissemotivicmaincomparisontheorem} then provides an alternative argument to the proof of \cite[Theorem~$7.7$]{elmanto_motivic_2023}.
	
	\begin{corollary}[Comparison to lisse motivic cohomology]\label{corollarylissemotivicmaincomparisontheorem}
		Let $R$ be a local ring. Then for every integer $i \geq 0$, the lisse-motivic comparison map induces a natural equivalence
		$$\Z(i)^{\emph{lisse}}(R) \xlongrightarrow{\sim} \tau^{\leq i} \Z(i)^{\emph{mot}}(R)$$
		in the derived category $\mathcal{D}(\Z)$.
	\end{corollary}
	
	\begin{proof}
		The classical motivic complex $\Z(i)^{\text{cla}}(-)$ is Zariski-locally in degrees at most $i$ (\cite[Corollary~$4.4$]{geisser_motivic_2004}). By \cite[Theorem~$5.10$]{bouis_motivic_2024} and the rational splitting of algebraic $K$-theory induced by Adams operations, the classical-motivic comparison map $\Z(i)^{\text{cla}}(-) \rightarrow \Z(i)^{\text{mot}}(-)$ is an isomorphism in degrees less than or equal to $i$ on smooth $\Z$-schemes.
        The result is then a consequence of Theorem~\ref{theoremmotiviccohomologyisleftKanextendedonlocalringsinsmalldegrees}.
	\end{proof}

	In the rest of this section, we restrict our attention to henselian local rings, in order to describe the motivic cohomology group $\text{H}^2_{\text{mot}}(-,\Z(1))$.
	
	\begin{lemma}\label{lemmavanishingpadici+1}
		Let $R$ be a henselian local ring, and $p$ be a prime number. Then for any integers $i \geq 0$ and $k \geq 1$, the motivic cohomology group $\emph{H}^{i+1}_{\emph{mot}}(R,\Z/p^k(i))$ is zero.
	\end{lemma}
	
	\begin{proof}
		By \cite[Theorem~$5.10$]{bouis_motivic_2024}, the motivic cohomology group $\text{H}^{i+1}_{\text{mot}}(R,\Z/p^k(i))$ is naturally identified with the kernel of the natural map of abelian groups
		$$\text{H}^{i+1}(\Z/p^k(i)^{\text{syn}}(R)) \longrightarrow \text{H}^{i+1}((L_{\text{cdh}} \tau^{>i} \Z/p^k(i)^{\text{syn}})(R))$$
		for every commutative ring $R$. If $R$ is henselian local, let $\mathfrak{m}$ be its maximal ideal, and consider the natural commutative diagram
		$$\begin{tikzcd}
			\text{H}^{i+1}(\Z/p^k(i)^{\text{syn}}(R)) \ar[r] \ar[d] & \text{H}^{i+1}((L_{\text{cdh}} \tau^{>i} \Z/p^k(i)^{\text{syn}})(R)) \ar[d] \\
			\text{H}^{i+1}(\Z/p^k(i)^{\text{syn}}(R/\mathfrak{m})) \ar[r] & \text{H}^{i+1}((L_{\text{cdh}} \tau^{>i} \Z/p^k(i)^{\text{syn}})(R/\mathfrak{m}))
		\end{tikzcd}$$
		of abelian groups. The functor $\tau^{>i} \Z/p^k(i)^{\text{syn}}$ is rigid (proof of Proposition~\ref{propositionmodpmotiviccohoisleftKanextendedinsmalldegrees}), so the left vertical map is an isomorphism. The field $R/\mathfrak{m}$ is a local ring for the cdh topology, so the bottom horizontal map is an isomorphism. In particular, the top horizontal map is injective.
	\end{proof}

	\begin{proposition}\label{propositionweightoneH3(1)vanishes}
		Let $R$ be a henselian local ring. Then for every integer $i \geq 1$, the motivic cohomology group $\emph{H}^{i+2}_{\emph{mot}}(R,\Z(i))$ is zero.
	\end{proposition}

	\begin{proof}
		By Lemma~\ref{lemmavanishingpadici+1} and the short exact sequence of abelian groups
		$$0 \longrightarrow \text{H}^{i+1}_{\text{mot}}(R,\Z(i))/p \longrightarrow \text{H}^{i+1}_{\text{mot}}(R,\F_p(i)) \longrightarrow \text{H}^{i+2}_{\text{mot}}(R,\Z(i))[p] \longrightarrow 0$$
		for every prime number $p$, the abelian group $\text{H}^{i+2}_{\text{mot}}(R,\Z(i))$ is torsionfree. By Proposition~\ref{propositionrationalcomparisonlissemotivicforlocalrings} if $i \geq 2$, and by Remark~\ref{remarkDrinfeldtheoremK-1onhenselianlocalrings} if $i=1$, it is also torsion, so it is zero.
	\end{proof}

	The following example is a consequence of \cite[Example~$3.9$]{bouis_motivic_2024}, Corollary~\ref{corollarylissemotivicmaincomparisontheorem}, and Proposition~\ref{propositionweightoneH3(1)vanishes}.

	\begin{example}\label{exampleweightonemotiviccohomology}
		For every qcqs scheme $X$, the natural map
		$$R\Gamma_{\text{Nis}}(X,\mathbb{G}_m)[-1] \longrightarrow \Z(1)^{\text{mot}}(X),$$
		defined as the Nisnevich sheafification of the lisse-motivic comparison map (see also Definition~\ref{definitionmotivicfirstChernclass}), is an isomorphism in degrees at most three. That is, the motivic complex $\Z(1)^{\text{mot}}(X)$ vanishes in degrees at most zero, and there are natural isomorphisms of abelian groups
		$$\text{H}^1_{\text{mot}}(X,\Z(1)) \cong \mathcal{O}(X)^{\times}, \quad \text{H}^2_{\text{mot}}(X,\Z(1)) \cong \text{Pic}(X), \quad \text{H}^3_{\text{mot}}(X,\Z(1)) \cong \text{H}^2_{\text{Nis}}(X,\mathbb{G}_m).$$
	\end{example}
	
	\subsection{Comparison to Milnor \texorpdfstring{$K$}{TEXT}-theory}
	
	\vspace{-\parindent}
	\hspace{\parindent}
	
	In this subsection, we construct, for every integer $i \geq 1$, a symbol map
	$$\text{K}^{\text{M}}_i(A) \longrightarrow \text{H}^i_{\text{mot}}(A,\Z(i))$$
	for local rings $A$ (Definition~\ref{definitionsymbolmap}), through which we compare the Milnor $K$-groups to motivic cohomology (Theorem~\ref{theoremcomparisontoMilnorKtheory}). Note that the arguments in Lemmas~\ref{lemmasymbolmapmilnorKtheorywelldefined} and \ref{lemmasymbolmapfactorsthoughimprovedMilnorKgroup} are very similar to that of \cite[Section~$7$]{elmanto_motivic_2023}, except for the Gersten injectivity for classical motivic cohomology, which is unknown integrally in mixed characteristic.
	
	For every commutative ring $A$, the lisse-motivic comparison map (Definition~\ref{definitionlissemotiviccomparisonmap} and \cite[Example~$3.9$]{bouis_motivic_2024}) induces on $\text{H}^1$ a natural isomorphism of abelian groups
	$$A^\times \xlongrightarrow{\cong} \text{H}^1_{\text{mot}}(A,\Z(1)).$$
	By multiplicativity of the motivic complexes, this induces, for every integer $i \geq 0$, a symbol map of abelian groups
	$$(A^{\times})^{\otimes i} \longrightarrow \text{H}^i_{\text{mot}}(A,\Z(i)).$$
	
	\begin{lemma}\label{lemmaMilnorK_2}
		For every local essentially smooth $\Z$-algebra $A$, there is a natural isomorphism
		$$\emph{H}^2_{\emph{mot}}(A,\Z(2)) \cong \emph{K}_2(A)$$
		of abelian groups.
	\end{lemma}

	\begin{proof}
		By Corollary~$2.12$, the classical-motivic comparison map
		$$\text{H}^2_{\text{cla}}(A,\Z(2)) \longrightarrow \text{H}^2_{\text{mot}}(A,\Z(2))$$
		is an isomorphism of abelian groups. The result is then a consequence of the Atiyah--Hirzebruch spectral sequence for classical motivic cohomology (\cite[Remark~$3.3$]{bouis_motivic_2024}), where we use that the classical motivic complex $\Z(1)^{\text{cla}}(A) \in \mathcal{D}(\Z)$ is concentrated in degree one (\cite[Example~$3.5$]{bouis_motivic_2024} and \cite[Corollary~$4.4$]{geisser_motivic_2004}).
	\end{proof}
	
	\begin{lemma}\label{lemmasymbolmapmilnorKtheorywelldefined}
		Let $A$ be a local ring. Then for every integer $i \geq 0$, the natural map of abelian groups
		$$(A^{\times})^{\otimes i} \longrightarrow \emph{H}^i_{\emph{mot}}(A,\Z(i))$$
		induced by the lisse-motivic comparison map factors through the Milnor $K$-group $\emph{K}^{\emph{M}}_i(A)$.
	\end{lemma}
	
	\begin{proof}
		By definition of the Milnor $K$-groups, it suffices to prove that the symbol map respects the Steinberg relations. Let $a \in A$ be an element such that $a$ and $1-a$ are units in $A$. By multiplicativity of the motivic complexes, it suffices to consider the case $i=2$ and to prove that $a \otimes (1-a)$ is sent to zero via the symbol map. Let $\Z[t] \rightarrow A$ be the ring homomorphism sending $t$ to $a$, and let $\mathfrak{p} \subset \Z[t]$ be the prime ideal defined as the inverse image of the maximal ideal of $A$ via this ring homomorphism. By naturality of the symbol map, the diagram of abelian groups
		$$\begin{tikzcd}
			(\Z[t]_{\mathfrak{p}})^{\times} \otimes_{\Z} (\Z[t]_{\mathfrak{p}})^{\times} \ar[r] \ar[d] & \text{H}^2_{\text{mot}}(\Z[t]_{\mathfrak{p}},\Z(2)) \ar[d] \\
			A^{\times} \otimes_{\Z} A^{\times} \ar[r] & \text{H}^2_{\text{mot}}(A,\Z(2))
		\end{tikzcd}$$
		is commutative. It then suffices to prove that the top horizontal arrow of this diagram sends $t \otimes (1-t)$ to zero. The local ring $\Z[t]_{\mathfrak{p}}$ is essentially smooth over $\Z$, so the right vertical map of the commutative diagram of abelian groups
		$$\begin{tikzcd}
			(\Z[t]_{\mathfrak{p}})^{\times} \otimes_{\Z} (\Z[t]_{\mathfrak{p}})^{\times} \ar[r] \ar[d] & \text{H}^2_{\text{mot}}(\Z[t]_{\mathfrak{p}},\Z(2)) \ar[d] \\
			(\text{Frac}(\Z[t]_{\mathfrak{p}}))^{\times} \otimes_{\Z} (\text{Frac}(\Z[t]_{\mathfrak{p}}))^{\times} \ar[r] & \text{H}^2_{\text{mot}}(\text{Frac}(\Z[t]_{\mathfrak{p}}),\Z(2))
		\end{tikzcd}$$
		is injective. Indeed, by Lemma~\ref{lemmaMilnorK_2} this is equivalent to the fact that the natural map
		$$\text{K}_2(A) \longrightarrow \text{K}_2(\text{Frac}(A))$$
		is injective, and this is the Gersten injectivity for $\text{K}_2$ (\cite[Corollary~$6$]{gillet_relative_1987} and \cite[Theorem~$2.2$]{dennis_K2_1975}). It then suffices to prove that the bottom horizontal map of this diagram sends $t \otimes (1-t)$ to zero. This is a consequence of the fact that the symbol map to classical motivic cohomology respects the Steinberg relation for fields (\cite{nesterenko_homology_1989}, see also \cite{totaro_milnor_1992}). 
	\end{proof}
	
	\begin{definition}[Symbol map]\label{definitionsymbolmap}
		Let $A$ be a local ring. For every integer $i \geq 0$, the {\it symbol map}
		$$\text{K}^{\text{M}}_{\text{i}}(A) \longrightarrow \text{H}^i_{\text{mot}}(A,\Z(i))$$
		is the natural map of abelian groups of Lemma~\ref{lemmasymbolmapmilnorKtheorywelldefined}.
	\end{definition}
	
	Following \cite{kerz_milnor_2010}, for $A$ a local ring and $i \geq 0$ an integer, we denote by $\widehat{\text{K}}{}^{\text{M}}_i(A)$ the $i^{\text{th}}$ {\it improved Milnor $K$-group} of $A$.
	
	\begin{lemma}\label{lemmasymbolmapfactorsthoughimprovedMilnorKgroup}
		Let $A$ be a local ring. Then for every integer $i \geq 0$, the symbol map
		$$\emph{K}^{\emph{M}}_i(A) \longrightarrow \emph{H}^i_{\emph{mot}}(A,\Z(i))$$
		factors through the improved Milnor $K$-group $\widehat{\emph{K}}{}^{\emph{M}}_i(A)$.
	\end{lemma}
	
	\begin{proof}
		Let $M_i \geq 1$ be the integer defined in \cite{kerz_milnor_2010}. If the residue field of the local ring $A$ has at least $M_i$ elements, then the natural map $$\text{K}^{\text{M}}_i(A) \longrightarrow \widehat{\text{K}}{}^{\text{M}}_i(A)$$
		is an isomorphism of abelian groups (\cite[Proposition~$10\,(5)$]{kerz_milnor_2010}). Assume now that the residue field of the local ring $A$ has less than $M_i$ elements. We want to prove that the symbol map
		$$\text{K}{}^{\text{M}}_i(A) \longrightarrow \text{H}^i_{\text{mot}}(A,\Z(i))$$
		factors through the surjective map $\text{K}{}^{\text{M}}_i(A) \rightarrow \widehat{\text{K}}{}^{\text{M}}_i(A)$, {\it i.e.}, that every element of the abelian group $\text{ker}(\text{K}{}^{\text{M}}_i(A) \rightarrow \widehat{\text{K}}{}^{\text{M}}_i(A))$ is sent to zero by the previous symbol map. Let $\mathfrak{m}$ be the maximal ideal of the local ring $A$, and $p$ be its residue characteristic. The residue field~$A/\mathfrak{m}$ of the local ring $A$ is isomorphic to a finite extension $\F_q$ of $\F_p$. Let $\ell \geq 1$ be an integer which is coprime to the degree of this extension, and such that $p^{\ell} \geq M_i$. As a tensor product of finite field extensions of coprime degree, the commutative ring $\F_q \otimes_{\F_p} \F_{p^{\ell}}$ is a field. Let $V$ be the finite étale extension of $\Z_{(p)}$ corresponding to the field extension $\F_{p^\ell}$ of $\F_p$. The commutative ring $A' := A \otimes_{\Z_{(p)}} V$ is finite over the local ring $A$, and the quotient $A'/\mathfrak{m}A'$ is a field, so the commutative ring $A'$ is a local $A$-algebra, whose residue field has at least $M_i$ elements.
		
		Let $P_{\bullet} \rightarrow A$ be a simplicial resolution of the local ring $A$ where each term $P_m$ is a local ind-smooth $\Z$-algebra, and each face map $P_{m+1} \rightarrow P_m$ is a henselian surjection. By Theorem~\ref{theoremmotiviccohomologyisleftKanextendedonlocalringsinsmalldegrees}, there is then a natural equivalence
		$$\underset{m}{\text{colim}}\, \tau^{\leq i} \Z(i)^{\text{mot}}(P_m) \xlongrightarrow{\sim} \tau^{\leq i} \Z(i)^{\text{mot}}(A)$$
		in the derived category $\mathcal{D}(\Z)$. In particular, this equivalence induces a natural isomorphism
		$$\text{coeq}\big(\text{H}^i_{\text{mot}}(P_1,\Z(i)) \hspace{1mm} \substack{\longrightarrow\\[-0.9em] \longrightarrow} \hspace{1mm} \text{H}^i_{\text{mot}}(P_0,\Z(i))\big) \xlongrightarrow{\cong} \text{H}^i_{\text{mot}}(A,\Z(i))$$
		of abelian groups, where the motivic cohomology groups in the left term are naturally identified with classical motivic cohomology groups by Corollary~\ref{corollarylissemotivicmaincomparisontheorem}. Similarly, $P_{\bullet} \otimes_{\Z_{(p)}} V \rightarrow A'$ is a simplicial resolution of the local ring $A'$ where each term $P_m \otimes_{\Z_{(p)}} V$ is an ind-smooth $V$\nobreakdash-algebra, and each face map $P_{m+1} \otimes_{\Z_{(p)}} V \rightarrow P_m \otimes_{\Z_{(p)}} V$ is a henselian surjection, so there is a natural isomorphism
		$$\text{coeq}\big(\text{H}^i_{\text{mot}}(P_1 \otimes_{\Z_{(p)}} V,\Z(i)) \hspace{1mm} \substack{\longrightarrow\\[-0.9em] \longrightarrow} \hspace{1mm} \text{H}^i_{\text{mot}}(P_0 \otimes_{\Z_{(p)}} V,\Z(i))\big) \xlongrightarrow{\cong} \text{H}^i_{\text{mot}}(A',\Z(i))$$
		of abelian groups, where the motivic cohomology groups of the left term are naturally identified with classical motivic cohomology groups. Classical motivic cohomology of smooth schemes over a mixed characteristic Dedekind domain admits functorial transfer maps along finite étale morphisms, so the previous two isomorphisms induce a transfer map
		$$N_{\ell} : \text{H}^i_{\text{mot}}(A',\Z(i)) \longrightarrow \text{H}^i_{\text{mot}}(A,\Z(i))$$
		such that pre-composition with the natural map $\text{H}^i_{\text{mot}}(A,\Z(i)) \rightarrow \text{H}^i_{\text{mot}}(A',\Z(i))$ is multiplication by~$\ell$. In particular, the kernel of the natural map $\text{H}^i_{\text{mot}}(A,\Z(i)) \rightarrow \text{H}^i_{\text{mot}}(A',\Z(i))$ is $\ell$-torsion.
		
		Consider the commutative diagram
		$$\begin{tikzcd}
			\text{K}{}^{\text{M}}_i(A) \ar[r] \ar[d] & \text{K}{}^{\text{M}}_i(A') \ar[d] \\
			\text{H}^i_{\text{mot}}(A,\Z(i)) \ar[r] & \text{H}^i_{\text{mot}}(A',\Z(i))
		\end{tikzcd}$$
		of abelian groups, and let $x$ be an element of the abelian group $\text{ker}(\text{K}{}^{\text{M}}_i(A) \rightarrow \widehat{\text{K}}{}^{\text{M}}_i(A))$. The residue field of the local ring $A'$ has at least $M_i$ elements, so the natural map $\text{K}{}^{\text{M}}_i(A') \rightarrow \widehat{\text{K}}{}^{\text{M}}_i(A')$ is an isomorphism (\cite[Proposition~$10\,(5)$]{kerz_milnor_2010}), and $x$ is sent to zero by the top horizontal map. In particular, the image of $x$ by the left vertical map is in the kernel of the bottom horizontal map, and is thus $\ell$-torsion by the previous paragraph. Let $\ell' \geq 1$ be an integer which is coprime to $\ell$ and to the degree of $\F_q$ over~$\F_p$, and such that $p^{\ell'} \geq M_i$. The previous argument for this integer $\ell'$ implies that the image of $x$ by the left vertical map is also $\ell'$-torsion, hence it is zero.
	\end{proof}
	
	\begin{conjecture}\label{conjecturecomparisonMilnorKtheoryandmotiviccohomology}
		Let $A$ be a local ring. Then for every integer $i \geq 0$, the natural map
		$$\widehat{\text{K}}{}^{\text{M}}_i(A) \longrightarrow \text{H}^i_{\text{mot}}(A,\Z(i))$$
		induced by Lemma~\ref{lemmasymbolmapfactorsthoughimprovedMilnorKgroup} is an isomorphism of abelian groups.
	\end{conjecture}
	
	The previous conjecture was proved by Elmanto--Morrow for equicharacteristic local rings (\cite[Theorem~$7.12$]{elmanto_motivic_2023}). Their proof uses as an input the analogous result in the smooth case for classical motivic cohomology, which is unknown in mixed characteristic (see Remark~\ref{remarksmoothcaseimpliesgeneralcaseMilnorcomparison}).
	
	\begin{theorem}[Singular Nesterenko--Suslin isomorphism with finite coefficients]\label{theoremcomparisontoMilnorKtheory}
		Let $A$ be a henselian local ring. Then for any integers $i \geq 0$ and $n \geq 1$, the natural map
		$$\widehat{\emph{K}}{}_i^{\emph{M}}(A)/n \longrightarrow \emph{H}^i_{\emph{mot}}(A,\Z(i))/n$$
		is an isomorphism of abelian groups.
	\end{theorem}
	
	\begin{proof}
		If the local ring $A$ contains a field, then the natural map
		$$\widehat{\text{K}}{}^{\text{M}}_i(A) \longrightarrow \text{H}^i_{\text{mot}}(A,\Z(i))$$
		is an isomorphism of abelian groups (\cite[Theorem~$7.12$]{elmanto_motivic_2023}). Otherwise, $A$ is a henselian local ring of mixed characteristic $(0,p)$ for some prime number $p$. In particular, the local ring~$A$ is $p$-henselian. If $p$ does not divide the integer $n$, then consider the commutative diagram
		$$\begin{tikzcd}
			\widehat{\text{K}}{}^{\text{M}}_i(A)/n \ar[r] \ar[d] & \text{H}^i_{\text{mot}}(A,\Z(i))/n \ar[d] \\
			\widehat{\text{K}}{}^{\text{M}}_i(A/p)/n \ar[r] & \text{H}^i_{\text{mot}}(A/p,\Z(i))/n
		\end{tikzcd}$$
		of abelian groups. The left vertical map is an isomorphism by \cite[Proposition~$10\,(7)$]{kerz_milnor_2010}. By Lemma~\ref{lemmatopdegreeintegraltomodpissurjective} and \cite[Corollary~$5.6$]{bouis_motivic_2024}, the right vertical map is naturally identified with the natural map of abelian groups
		$$\text{H}^i_{\text{ét}}(A,\mu_n^{\otimes i}) \longrightarrow \text{H}^i_{\text{ét}}(A/p,\mu_n^{\otimes i}),$$
		which is an isomorphism by rigidity of étale cohomology (\cite{gabber_affine_1994}, see also \cite[Corollary~$1.18\,(1)$]{bhatt_arc-topology_2021}). The local ring $A/p$ is an $\F_p$-algebra, so the bottom horizontal map is an isomorphism (\cite[Theorem~$7.12$]{elmanto_motivic_2023}), and the result is true in this case. It then suffices to prove that for every integer $k \geq 1$, the natural map
		$$\widehat{\text{K}}{}^{\text{M}}_i(A)/p^k \longrightarrow \text{H}^i_{\text{mot}}(A,\Z(i))/p^k$$
		is an isomorphism of abelian groups. By Lemma~\ref{lemmatopdegreeintegraltomodpissurjective} and \cite[Corollary~$5.11$]{bouis_motivic_2024}, the natural map
		$$\text{H}^i_{\text{mot}}(A,\Z(i))/p^k \longrightarrow \text{H}^i(\Z/p^k(i)^{\text{BMS}}(A))$$
		is an isomorphism of abelian groups. By \cite[Theorem~$3.1$]{luders_milnor_2023}, the composite map
		$$\widehat{\text{K}}{}^{\text{M}}_i(A)/p^k \longrightarrow \text{H}^i_{\text{mot}}(A,\Z(i))/p^k \longrightarrow \text{H}^i(\Z/p^k(i)^{\text{BMS}}(A))$$
		is an isomorphism of abelian groups, hence the desired result.
	\end{proof}
	
	\begin{remark}\label{remarksmoothcaseimpliesgeneralcaseMilnorcomparison}
		If Conjecture~\ref{conjecturecomparisonMilnorKtheoryandmotiviccohomology} is true for local ind-smooth $\Z$-algebras, then the left Kan extension properties \cite[Proposition~$1.17$]{luders_milnor_2023} and Theorem~\ref{theoremmotiviccohomologyisleftKanextendedonlocalringsinsmalldegrees} imply that Conjecture~\ref{conjecturecomparisonMilnorKtheoryandmotiviccohomology} is true for all local rings. See the proof of \cite[Theorem~$7.12$]{elmanto_motivic_2023} for more details.
	\end{remark}
	
	\begin{remark}\label{remarkMilnorcomparisonisinjectiveiffGersteninjectivitytrueforMilnorKgroupssmoothcase}
		Let $A$ be a local essentially smooth $\Z$-algebra, $i \geq 0$ be an integer, and consider the commutative diagram
		$$\begin{tikzcd}
			\widehat{\text{K}}{}^{\text{M}}_i(A) \ar[r] \ar[d] & \text{H}^i_{\text{mot}}(A,\Z(i)) \ar[d] \\
			\widehat{\text{K}}{}^{\text{M}}_i(\text{Frac}(A)) \ar[r] & \text{H}^i_{\text{mot}}(\text{Frac}(A),\Z(i))
		\end{tikzcd}$$
		of abelian groups. The bottom horizontal map is an isomorphism by the Nesterenko--Suslin isomorphism for fields (\cite{nesterenko_homology_1989}, see also \cite[Theorem~$7.12$]{elmanto_motivic_2023}). The left vertical being injective then implies that the top horizontal map is injective. That is, the Gersten injectivity conjecture for the improved Milnor $K$-groups would imply the injectivity part of Conjecture~\ref{conjecturecomparisonMilnorKtheoryandmotiviccohomology}. Knowing the Gersten injectivity conjecture for the motivic cohomology group $\text{H}^i_{\text{mot}}(-,\Z(i))$ would imply that these two facts are equivalent. See \cite{luders_relative_2022} for related results on the Gersten conjecture for improved Milnor $K$-groups.
	\end{remark}

	\section{Weibel vanishing and pro cdh descent}\label{sectionweibelprocdh}
	
	\vspace{-\parindent}
	\hspace{\parindent}
	
	In this section, we study the motivic complexes $\Z(i)^{\text{mot}}$ on noetherian schemes. We prove a general vanishing result which refines Weibel's vanishing conjecture on negative $K$-groups (Theorem~\ref{theoremmotivicWeibelvanishing}), and prove that they coincide with Kelly--Saito's pro cdh motivic complexes $\Z(i)^{\text{procdh}}$ (Theorem~\ref{theoremcomparisonprocdhmotivic}). The key input for both these results is the fact that the motivic complexes $\Z(i)^{\text{mot}}$ satisfy pro cdh excision (Theorem~\ref{theoremprocdhdescentformotiviccohomology}), {\it i.e.}, that they send abstract blowup squares to pro cartesian squares.
	
	\begin{notation}[Abstract blowup square]\label{notationabstractblowupsquare}
		An {\it abstract blowup square} (of noetherian schemes) is a cartesian square
		\begin{equation}\label{equationabstractblowupsquare}
			\begin{tikzcd}
				Y' \ar[r] \ar[d] & X' \ar[d] \\
				Y \ar[r] & X
			\end{tikzcd}
		\end{equation}
		of qcqs schemes (resp. of noetherian schemes)
		such that $X' \rightarrow X$ is proper and finitely presented, \hbox{$Y \rightarrow X$} is a finitely presented closed immersion, and the induced map \hbox{$X'\setminus Y' \rightarrow X \setminus Y$} is an isomorphism. In this context, we also denote, for every integer $r \geq 0$, by $rY$ (resp.~$rY'$) the $r-1^{\text{st}}$ infinitesimal thickening of $Y$ inside $X$ (resp. of $Y'$ inside $X'$).
	\end{notation}	
	
	We use Kelly--Saito's recent definition in \cite{kelly_procdh_2024} of the pro cdh topology to encode the fact that a Nisnevich sheaf ({\it e.g.}, the motivic complex $\Z(i)^{\text{mot}}$) satisfies pro cdh excision. Kelly--Saito proved in particular that if $S$ has finite valuative dimension and noetherian topological space, then the pro cdh topos of $S$ is hypercomplete and has enough points. For our purposes, the following definition will be used only for noetherian schemes $S$.
	
	\begin{definition}[Pro cdh descent, after \cite{kelly_procdh_2024}]\label{definitionprocdhsheaf}
		Let $S$ be a qcqs scheme. A {\it pro cdh sheaf} on finitely presented $S$-schemes is a presheaf
		$$F : \text{Sch}_S^{\text{fp,op}} \longrightarrow \mathcal{D}(\Z)$$
		satisfying Nisnevich descent, and such that for every abstract blowup square of finitely presented $S$-schemes $(\ref{equationabstractblowupsquare})$, the natural commutative diagram
		$$\begin{tikzcd}
			F(X) \ar[r] \ar[d] & F(X') \ar[d] \\
			\{F(rY)\}_r \ar[r] & \{F(rY')\}_r
		\end{tikzcd}$$
		is a weakly cartesian square of pro objects in the derived category $\mathcal{D}(\Z)$.\footnote{By this, we mean that all the cohomology groups of the total fibre of this commutative square are zero as pro abelian groups. All the presheaves $F$ that we will consider (most importantly, the presheaves $\Z(i)^{\text{mot}}$) are bounded above on noetherian schemes, by a constant depending only on the dimension of their input (for the motivic complexes $\Z(i)^{\text{mot}}$, this is \cite[Proposition~$4.49$]{bouis_motivic_2024}); this definition of weakly cartesian square will then be equivalent to being weakly cartesian in the stable $\infty$-category of pro objects in the derived category $\mathcal{D}(\Z)$, in the sense of \cite[Definition~$2.27$]{land_k-theory_2019}.}
	\end{definition}
	
	\subsection{Pro cdh descent for the cotangent complex}
	
	\vspace{-\parindent}
	\hspace{\parindent}
	
	In this subsection, we review the pro cdh descent for powers of the cotangent complex on noetherian schemes (Proposition~\ref{propositionprocdhdescentforcotangentcomplex}). On finite-dimensional noetherian schemes, this is \cite[Theorem~$2.10$]{morrow_pro_2016}. On general noetherian schemes, the proof follows the sketch presented in \cite[proof of Lemma~$8.5$]{elmanto_motivic_2023}. In particular, the arguments are exactly as in \cite{morrow_pro_2016}, except for the following generalisation of Grothendieck's formal functions theorem (\cite[Corollary~$4.1.7$]{grothendieck_elements_1961}), where the finite dimensionality hypothesis is removed. We give some details for the sake of completeness.
	
	For every commutative ring $A$, recall that a pro $A$-module $\{M_r\}_r$ is zero if for every index~$r$, there exists an index $r' \geq r$ such that the map $M_{r'} \rightarrow M_r$ is the zero map. Similarly, a map $\{M_r\}_r \rightarrow \{N_r\}_r$ of pro $A$-modules is an isomorphism if its kernel and cokernel are zero pro $A$-modules. We say that a pro object $\{C_r\}_r$ in the derived category $\mathcal{D}(A)$ is weakly zero if all its cohomology groups are zero pro $A$-modules. Note that all the pro complexes that we will consider are uniformly bounded above, so this definition is equivalent to being weakly zero in the stable $\infty$-category of pro objects in the derived category $\mathcal{D}(A)$ (\cite[Definition~$2.27$]{land_k-theory_2019}). Similarly, we say that a map $\{C_r\}_r \rightarrow \{C_r'\}_r$ of pro objects in the derived category $\mathcal{D}(A)$ is a weak equivalence if its fibre is weakly zero as a pro object in the derived category $\mathcal{D}(A)$.
	
	\begin{lemma}[Formal functions theorem, after Lurie \cite{lurie_spectral_2019}]\label{lemmaformalfunctionstheoremLurie}
		Let $A$ be a noetherian commutative ring, $I$ be an ideal of $A$, $X$ be a proper scheme over $\emph{Spec}(A)$, and $X^\wedge_I$ be the formal completion of $X$ along the vanishing locus of $I$. Then for every coherent sheaf $\mathcal{F}$ over $X$, the natural map
		$$R\Gamma_{\emph{Zar}}(X,\mathcal{F}) \longrightarrow R\Gamma_{\emph{Zar}}(X^\wedge_I,\mathcal{F}^\wedge_I)$$
		where $\mathcal{F}^\wedge_I$ is the pullback of $\mathcal{F}$ along the natural map $X^\wedge_I \rightarrow X$, exhibits the target as the $I$-adic completion\footnote{The cohomology groups of these coherent sheaves are finitely generated $A$-modules (because $X$ is proper over $\text{Spec}(A)$), so the derived $I$-adic completion and the classical $I$-adic completion coincide in this context.} of the source in the derived category $\mathcal{D}(A)$. More precisely, the natural map
		$$\{R\Gamma_{\emph{Zar}}(X,\mathcal{F})/I^r\}_r \longrightarrow \{R\Gamma_{\emph{Zar}}(X \times_{\emph{Spec}(A)} \emph{Spec}(A/I^r), \mathcal{F} \otimes^{\mathbb{L}}_{\mathcal{O}_X} \mathcal{O}_X/I^r\mathcal{O}_X)\}_r$$
		is a weak equivalence of pro objects in the derived category $\mathcal{D}(A)$.
	\end{lemma}
	
	\begin{proof}
		The first statement is a special case of \cite[Lemma~$8.5.1.1$]{lurie_spectral_2019}. The second statement, although {\it a priori} stronger, follows by an examination of the previous proof (and in particular, the proof of \cite[Lemma~$8.1.2.3$]{lurie_spectral_2019}).
	\end{proof}
	
	\begin{lemma}\label{lemmacohomologygroupsofcotangentcomplexarefinitelygenerated}
		Let $A$ be a noetherian commutative ring, and $X$ be a proper scheme over $\emph{Spec}(A)$. Then for any integers $j \geq 0$ and $n \in \Z$, the $A$-module $\emph{H}^n_{\emph{Zar}}(X,\mathbb{L}^j_{-/A})$ is finitely generated.
	\end{lemma}
	
	\begin{proof}
		The scheme $X$ is of finite type over $\text{Spec}(A)$, so the $\mathcal{O}_X$-module $\mathcal{H}^n_{\text{Zar}}(-,\mathbb{L}^j_{-/A})$ is coherent. Because $X$ is proper over $\text{Spec}(A)$, its cohomology groups are thus finitely generated $A$-modules.
	\end{proof}
	
	\begin{corollary}\label{corollaryM16Theorem2.4(iv)}
		Let $A$ be a noetherian commutative ring, $I$ be an ideal of $A$, and $X$ be a noetherian scheme which is proper over $\emph{Spec}(A)$. Then for any integers $j \geq 0$ and $n \in \Z$, the natural map
		$$\{ \emph{H}^n_{\emph{Zar}}(X,\mathbb{L}^j_{-/A})/I^r\}_r \longrightarrow \{\emph{H}^n_{\emph{Zar}}(X,\mathbb{L}^j_{-/A} \otimes^{\mathbb{L}}_{\mathcal{O}_X}\mathcal{O}_X / I^r \mathcal{O}_X) \}_r$$
		is an isomorphism of pro $A$-modules.
	\end{corollary}
	
	\begin{proof}
		By Lemma~\ref{lemmacohomologygroupsofcotangentcomplexarefinitelygenerated} and its proof, all the terms in the hypercohomology spectral sequence
		$$E_2^{p,q} = \text{H}^p_{\text{Zar}}(X,\mathcal{H}^q_{\text{Zar}}(-,\mathbb{L}^j_{-/R})) \Longrightarrow \text{H}^{p+q}_{\text{Zar}}(X,\mathbb{L}^j_{-/A})$$
		are finitely generated $A$-modules. The functor $\{-\otimes_A A/I^r\}_r$ is exact on the category of finitely generated $A$-modules (\cite[Theorem~$1.1$\,(ii)]{morrow_pro_2016}), so it induces a spectral sequence of pro $A$-modules
		$$E_2^{p,q} = \{\text{H}^p_{\text{Zar}}(X,\mathcal{H}^q_{\text{Zar}}(-,\mathbb{L}^j_{-/A}))/I^r\}_r \Longrightarrow \{\text{H}^{p+q}_{\text{Zar}}(X,\mathbb{L}^j_{-/A})/I^r\}_r.$$
		It then suffices to prove that the natural map of pro $A$-modules
		$$\{\text{H}^p_{\text{Zar}}(X,\mathcal{H}^q_{\text{Zar}}(-,\mathbb{L}^j_{-/A}))/I^r\}_r \longrightarrow \{\text{H}^p_{\text{Zar}}(X,\mathcal{H}^q_{\text{Zar}}(-,\mathbb{L}^j_{-/A} \otimes^{\mathbb{L}}_{\mathcal{O}_X} \mathcal{O}_X/I^r\mathcal{O}_X))\}_r$$
		is an isomorphism for all integers $p,q \geq 0$. The natural map 
		$$\{\text{H}^p_{\text{Zar}}(X,\mathcal{H}^q_{\text{Zar}}(-,\mathbb{L}^j_{-/A}))/I^r\}_r \longrightarrow \{\text{H}^p_{\text{Zar}}(X,\mathcal{H}^q_{\text{Zar}}(-,\mathbb{L}^j_{-/A}) \otimes_A^{\mathbb{L}} A/I^r)\}_r$$
		is an isomorphism by Lemma~\ref{lemmaformalfunctionstheoremLurie} applied to the coherent sheaf $\mathcal{H}^q_{\text{Zar}}(-,\mathbb{L}_{-/A})$ on $X$, and the natural map
		$$\{\text{H}^p_{\text{Zar}}(X,\mathcal{H}^q_{\text{Zar}}(-,\mathbb{L}^j_{-/A}) \otimes_A^{\mathbb{L}} A/I^r)\}_r \longrightarrow \{\text{H}^p_{\text{Zar}}(X,\mathcal{H}^q_{\text{Zar}}(-,\mathbb{L}^j_{-/A} \otimes^{\mathbb{L}}_{\mathcal{O}_X} \mathcal{O}_X/I^r\mathcal{O}_X))\}_r$$
		is an isomorphism by \cite[Lemma~$2.3$]{morrow_pro_2016}.
	\end{proof}
	
	\begin{lemma}\label{lemmaLemma2.8(i)}
		Let $A$ be a noetherian commutative ring, $I$ be an ideal of $A$, and $X$ be a proper scheme over $\emph{Spec}(A)$ such that the induced map $X \setminus V(I\mathcal{O}_X) \rightarrow \emph{Spec}(A) \setminus V(I)$ is an isomorphism.
		\begin{enumerate}
			\item For any integers $j \geq 0$ and $n \in \Z$, the natural map
			$$\{\emph{H}^n_{\emph{Zar}}(X,\mathbb{L}^j_{-/A} \otimes^{\mathbb{L}}_{\mathcal{O}_X} I^r \mathcal{O}_X)\}_r \longrightarrow \{I^r\emph{H}^n_{\emph{Zar}}(X,\mathbb{L}^j_{-/A})\}_r$$
			is an isomorphism of pro $A$-modules.
			\item For any integers $j \geq 0$ and $n \in \Z$ such that $(j,n) \neq (0,0)$, the $A$-module $\emph{H}^n_{\emph{Zar}}(X,\mathbb{L}^j_{-/A})$ is killed by a power of $I$; in particular, the pro $A$-module $\{I^r\emph{H}^n_{\emph{Zar}}(X,\mathbb{L}^j_{-/A})\}_r$ is zero.
		\end{enumerate}
	\end{lemma}
	
	\begin{proof}
		$(1)$ The short exact sequence $0 \rightarrow \{I^r \mathcal{O}_X \}_r \rightarrow \mathcal{O}_X \rightarrow \{\mathcal{O}_X / I^r \mathcal{O}_X \}_r \rightarrow 0$ of pro $\mathcal{O}_X$-modules induces a long exact sequence
		$$\cdots \rightarrow \text{H}^n_{\text{Zar}}(X,\mathbb{L}^j_{-/A} \otimes^{\mathbb{L}}_{\mathcal{O}_X} I^r \mathcal{O}_X)\}_r \rightarrow \text{H}^n_{\text{Zar}}(X,\mathbb{L}^j_{-/A}) \rightarrow \{\text{H}^n_{\text{Zar}}(X,\mathbb{L}^j_{-/A} \otimes^{\mathbb{L}}_{\mathcal{O}_X} \mathcal{O}_X/I^r\mathcal{O}_X)\}_r \rightarrow \cdots$$
		of pro $A$-modules. By Corollary~\ref{corollaryM16Theorem2.4(iv)}, the boundary maps of this long exact sequence vanish, hence the natural map
		$$\{\text{H}^n_{\text{Zar}}(X,\mathbb{L}^j_{-/A} \otimes^{\mathbb{L}}_{\mathcal{O}_X} I^r \mathcal{O}_X)\}_r \longrightarrow \{I^r\text{H}^n_{\text{Zar}}(X,\mathbb{L}^j_{-/A})\}_r$$
		is an isomorphism of pro $A$-modules.
		\\$(2)$ By Lemma~\ref{lemmacohomologygroupsofcotangentcomplexarefinitelygenerated}, the $A$-module $\text{H}^n_{\text{Zar}}(X,\mathbb{L}^j_{-/A})$ is finitely generated. Because the map $$X \setminus V(I\mathcal{O}_X) \rightarrow \text{Spec}(A) \setminus V(I)$$ is an isomorphism, this $A$-module is moreover supported on $V(I)$ if $(j,n) \neq (0,0)$. If $(j,n) \neq (0,0)$, this implies that the $A$-module $\text{H}^n_{\text{Zar}}(X,\mathbb{L}^j_{-/A})$ is killed by a power of $I$.
	\end{proof}
	
	\begin{corollary}\label{corollaryM16lemma2.8(ii)}
		Let $A$ be a noetherian commutative ring, $I$ be an ideal of $A$, and $X$ be a proper scheme over $\emph{Spec}(A)$ such that the induced map $X \setminus V(I\mathcal{O}_X) \rightarrow \emph{Spec}(A) \setminus V(I)$ is an isomorphism. Then for every integer $j \geq 0$, the natural map
		$$\{ \mathbb{L}^j_{A/\Z} \otimes^{\mathbb{L}}_A I^r \}_r \longrightarrow \{ R\Gamma_{\emph{Zar}}(X,\mathbb{L}^j_{-/\Z} \otimes^{\mathbb{L}}_{\mathcal{O}_X} I^r \mathcal{O}_X) \}_r$$
		is a weak equivalence of pro objects in the derived category $\mathcal{D}(A)$.
	\end{corollary}
	
	\begin{proof}
		By Lemma~\ref{lemmaLemma2.8(i)}, and for any integers $n,a,b \in \Z$, the pro $A$-module $$\{\text{H}^n(\mathbb{L}^j_{A/\Z} \otimes^{\mathbb{L}}_A \text{H}^a_{\text{Zar}}(X,\mathbb{L}^b_{-/A} \otimes^{\mathbb{L}}_{\mathcal{O}_X} I^r \mathcal{O}_X))\}_r$$ is zero, except if $(a,b)=(0,0)$. By transitivity for the powers of the cotangent complex (see the proof of \cite[Lemma~$2.8\,(ii)$]{morrow_pro_2016} for more details), this implies that the natural map
		$$\{\text{H}^n(\mathbb{L}^j_{A/\Z} \otimes^{\mathbb{L}}_A \text{H}^0_{\text{Zar}}(X,I^r\mathcal{O}_X))\}_r \longrightarrow \{\text{H}^n_{\text{Zar}}(X,\mathbb{L}^j_{-/\Z} \otimes^{\mathbb{L}}_{\mathcal{O}_X} I^r \mathcal{O}_X)\}_r$$
		is an isomorphism of pro $A$-modules. Let $B$ be the $A$-algebra $\text{H}^0_{\text{Zar}}(X,\mathcal{O}_X)$. Applying Lemma~\ref{lemmaLemma2.8(i)}\,$(1)$ for $j=n=0$, it then suffices to prove that the natural map $\{I^r\}_r \rightarrow \{I^r B\}_r$ is an isomorphism of pro $A$-modules. The $A$-algebra $B$ is finite and isomorphic to $A$ away from the vanishing locus of $I$, so the kernel and cokernel of the structure map $A \rightarrow B$ are killed by a power of $I$. The result is then a formal consequence of \cite[Theorem~$1.1$\,(ii)]{morrow_pro_2016}.
	\end{proof}
	
	\begin{lemma}\label{lemmaM16Lemma2.4.(i)}
		Let $Y \rightarrow X$ be a closed immersion of noetherian schemes, and $\mathcal{I}$ be the associated ideal sheaf on $X$. Then for every integer $j \geq 0$, the natural map
		$$\{ R\Gamma_{\emph{Zar}}(X,\mathbb{L}^j_{-/\Z} \otimes^{\mathbb{L}}_{\mathcal{O}_X} \mathcal{O}_X/\mathcal{I}^r) \}_r \longrightarrow \{ R\Gamma_{\emph{Zar}}(rY,\mathbb{L}^j_{-/\Z}) \}_r$$
		is a weak equivalence of pro objects in the derived category $\mathcal{D}(A)$.
	\end{lemma}
	
	\begin{proof}
		The scheme $X$ is noetherian, hence quasi-separated, so we may assume by induction that $X$ is affine. In this case, the result is \cite[Corollary~$4.5\,(ii)$]{morrow_pro_2018}.
	\end{proof}
	
	\begin{proposition}\label{propositionprocdhdescentforcotangentcomplex}
		Let $j \geq 0$ be an integer. Then for every abstract blowup square of noetherian schemes (\ref{equationabstractblowupsquare}), the natural commutative diagram
		$$\begin{tikzcd}
			R\Gamma_{\emph{Zar}}(X,\mathbb{L}^j_{-/\Z}) \ar[r] \ar[d] & R\Gamma_{\emph{Zar}}(X',\mathbb{L}^j_{-/\Z}) \ar[d] \\
			\{R\Gamma_{\emph{Zar}}(rY,\mathbb{L}^j_{-/\Z})\}_r \ar[r] & \{R\Gamma_{\emph{Zar}}(rY',\mathbb{L}^j_{-/\Z})\}_r
		\end{tikzcd}$$
		is a weakly cartesian square of pro objects in the derived category $\mathcal{D}(\Z)$. In particular, the presheaf $R\Gamma_{\emph{Zar}}(-,\mathbb{L}^j_{-/\Z})$ is a pro cdh sheaf on noetherian schemes.
	\end{proposition}
	
	\begin{proof}
		The scheme $X$ is noetherian, hence quasi-separated, so we may assume by induction that $X$ is affine, given by the spectrum of a noetherian commutative ring $A$. Let $I$ be the ideal of $A$ defining the closed subscheme $Y$ of $\text{Spec}(A)$. By Lemma~\ref{lemmaM16Lemma2.4.(i)}, the desired statement is equivalent to the fact that the commutative diagram
		$$\begin{tikzcd}
			\mathbb{L}^j_{A/\Z} \ar[r] \ar[d] & R\Gamma_{\text{Zar}}(X',\mathbb{L}^j_{-/\Z}) \ar[d] \\
			\{\mathbb{L}^j_{A/\Z} \otimes^{\mathbb{L}}_A A/I^r\}_r \ar[r] & \{R\Gamma_{\text{Zar}}(X',\mathbb{L}^j_{-/\Z} \otimes^{\mathbb{L}}_{\mathcal{O}_{X'}} \mathcal{O}_{X'}/I^r\mathcal{O}_{X'})\}_r
		\end{tikzcd}$$
		is a weakly cartesian square of pro objects in the derived category $\mathcal{D}(\Z)$. Taking fibres along the vertical maps, this is exactly Corollary~\ref{corollaryM16lemma2.8(ii)}.
	\end{proof}
	
	We now use Proposition~\ref{propositionprocdhdescentforcotangentcomplex} to prove pro cdh descent for variants of the cotangent complex (Corollary~\ref{corollaryprocdhdescentforcotangentcomplexandvariants}). In the following two lemmas, we consider inverse systems of objects in the derived category~$\mathcal{D}(\Z)$. We say that an inverse system $(C_r)_r$ in the derived category $\mathcal{D}(\Z)$ is {\it essentially zero} if for every index $r$ and every integer $n \in \Z$, there exists an index $r' \geq r$ such that the map $\text{H}^n(C_{r'}) \rightarrow \text{H}^n(C_r)$ is the zero map. In particular, an inverse system $(C_r)_r$ in the derived category $\mathcal{D}(\Z)$ is essentially zero if and only if the associated pro object $\{C_r\}_r$ in the derived category $\mathcal{D}(\Z)$ is weakly zero.
	
	\begin{lemma}\label{lemmaprozeroimpliesprozeromodp}
		Let $(C_r)_r$ be an inverse system in the derived category $\mathcal{D}(\Z)$. If $(C_r)_r$ is essentially zero, then $\big(\prod_{p \in \mathbb{P}} C_r/p\big)_r$ is essentially zero.
	\end{lemma}
	
	\begin{proof}
		Assume that the inverse system $(C_r)_r$ is essentially zero. Let $r_0$ be an index of this inverse system, $n \in \Z$ be an integer, and $p$ be a prime number. We will use repeatedly that for every index $r$, there is a natural short exact sequence
		$$0 \longrightarrow \text{H}^n(C_r)/p \longrightarrow \text{H}^n(C_r/p) \longrightarrow \text{H}^{n+1}(C_r)[p] \longrightarrow 0$$
		of abelian groups. Let $r_1 \geq r_0$ be an index such that the map $\text{H}^n(C_{r_1}) \rightarrow \text{H}^n(C_{r_0})$ is the zero map. Then for every index $r \geq r_1$, the map $\text{H}^n(C_r)/p \rightarrow \text{H}^n(C_{r_0})/p$ is the zero map, and the map $\text{H}^n(C_r/p) \rightarrow \text{H}^n(C_{r_0}/p)$ thus factors through the map $\text{H}^n(C_r/p) \rightarrow \text{H}^{n+1}(C_r)[p]$. Let~\hbox{$r_2 \geq r_1$} be an index such that the map $\text{H}^{n+1}(C_{r_2}) \rightarrow \text{H}^{n+1}(C_{r_1})$ is the zero map. Then the map $$\text{H}^{n+1}(C_{r_2})[p] \rightarrow \text{H}^{n+1}(C_{r_1})[p]$$ is the zero map. By construction, the map $$\text{H}^n(C_{r_2}/p) \longrightarrow \text{H}^n(C_{r_0}/p)$$ factors as
		$$\text{H}^n(C_{r_2}/p) \longrightarrow \text{H}^{n+1}(C_{r_2})[p] \xlongrightarrow{0} \text{H}^{n+1}(C_{r_1})[p] \longrightarrow \text{H}^n(C_{r_0}/p),$$
		and is thus also the zero map. The index $r_2$ does not depend on the prime number $p$, so the map
		$$\prod_{p \in \mathbb{P}} \text{H}^n(C_{r_2}/p) \longrightarrow \prod_{p \in \mathbb{P}} \text{H}^n(C_{r_0}/p)$$
		is the zero map, and the inverse system $\big(\prod_{p \in \mathbb{P}} C_r/p\big)_r$ is essentially zero.
	\end{proof}
	
	\begin{lemma}\label{lemmaprozeroimpliesprozeropcompletion}
		Let $(C_r)_r$ be an inverse system in the derived category $\mathcal{D}(\Z)$. If $(C_r)_r$ is essentially zero, then $\big(\prod_{p \in \mathbb{P}} (C_r)^\wedge_p\big)_r$ is essentially zero.
	\end{lemma}
	
	\begin{proof}
		Assume that the inverse system $(C_r)_r$ is essentially zero. Let $r_0$ be an index of this inverse system, $n \in \Z$ be an integer, and $p$ be a prime number. We will use repeatedly that for every index $r \geq 0$, there is a short exact sequence
		$$0 \longrightarrow \text{Ext}^1_{\Z_p}(\Q_p/\Z_p, \text{H}^n(C_r)) \longrightarrow \text{H}^n((C_r)^\wedge_p) \longrightarrow \text{Hom}_{\Z_p}(\Q_p/\Z_p, \text{H}^{n+1}(C_r)) \longrightarrow 0$$
		of abelian groups. Let $r_1 \geq r_0$ be an index such that the map $\text{H}^n(C_{r_1}) \rightarrow \text{H}^n(C_{r_0})$ is the zero map. Then for every index $r \geq r_1$, the map $\text{Ext}^1_{\Z_p}(\Q_p/\Z_p, \text{H}^n(C_r)) \rightarrow \text{Ext}^1_{\Z_p}(\Q_p/\Z_p, \text{H}^n(C_{r_0}))$ is the zero map, and the map $\text{H}^n((C_r)^\wedge_p) \rightarrow \text{H}^n((C_{r_0})^\wedge_p)$ thus factors through the map $$\text{H}^n((C_r)^\wedge_p) \longrightarrow \text{Hom}_{\Z_p}(\Q_p/\Z_p, \text{H}^{n+1}(C_r)).$$ Let $r_2 \geq r_1$ be an index such that the map $\text{H}^{n+1}(C_{r_2}) \rightarrow \text{H}^{n+1}(C_{r_1})$ is the zero map. Then the map $\text{Hom}_{\Z_p}(\Q_p/\Z_p, \text{H}^{n+1}(C_{r_2})) \rightarrow \text{Hom}_{\Z_p}(\Q_p/\Z_p, \text{H}^{n+1}(C_{r_1}))$ is the zero map. By construction, the map $$\text{H}^n((C_{r_2})^\wedge_p) \longrightarrow \text{H}^n((C_{r_0})^\wedge_p)$$ factors as
		$$\text{H}^n((C_{r_2})^\wedge_p) \longrightarrow \text{Hom}_{\Z_p}(\Q_p/\Z_p, \text{H}^{n+1}(C_{r_2})) \xlongrightarrow{0} \text{Hom}_{\Z_p}(\Q_p/\Z_p, \text{H}^{n+1}(C_{r_1})) \longrightarrow \text{H}^n((C_{r_0})^\wedge_p),$$
		and is thus the zero map. The index $r_2$ does not depend on the prime number $p$, so the map
		$$\prod_{p \in \mathbb{P}} \text{H}^n((C_{r_2})^\wedge_p) \longrightarrow \prod_{p \in \mathbb{P}} \text{H}^n((C_{r_0})^\wedge_p)$$
		is the zero map, and the inverse system $\big(\prod_{p \in \mathbb{P}} (C_r)^\wedge_p\big)_r$ is essentially zero.
	\end{proof}	
	
	\begin{corollary}\label{corollaryprocdhdescentforcotangentcomplexandvariants}
		Let $j \geq 0$ be an integer, and let $F$ be one of the presheaves $$R\Gamma_{\emph{Zar}}\big(-,\mathbb{L}^j_{-/\Z}\big),\text{ } R\Gamma_{\emph{Zar}}\big(-,\prod_{p \in \mathbb{P}} \mathbb{L}^j_{-_{\F_p}/\F_p}\big), \text{ } R\Gamma_{\emph{Zar}}\big(-,\prod_{p \in \mathbb{P}}(\mathbb{L}^j_{-/\Z})^\wedge_p\big), \text{ and } R\Gamma_{\emph{Zar}}\big(-,\mathbb{L}^j_{-_{\Q}/\Q}\big),$$
		where $-_{\F_p}$ is the derived base change from $\Z$ to $\F_p$. Then the presheaf $F$ is a pro cdh sheaf on noetherian schemes.
	\end{corollary}
	
	\begin{proof}
		The presheaf $F$ is a Nisnevich sheaf, so the result is equivalent to proving that $F$ sends an abstract blowup square of noetherian schemes to a weakly cartesian square of pro objects in the derived category~$\mathcal{D}(\Z)$. For $R\Gamma_{\text{Zar}}\big(-,\mathbb{L}^j_{-/\Z}\big)$, this is Proposition~\ref{propositionprocdhdescentforcotangentcomplex}. For $R\Gamma_{\text{Zar}}\big(-,\prod_{p \in \mathbb{P}} \mathbb{L}^j_{-_{\F_p}/\F_p}\big)$, this is a formal consequence of Proposition~\ref{propositionprocdhdescentforcotangentcomplex} and Lemma~\ref{lemmaprozeroimpliesprozeromodp}. For $R\Gamma_{\text{Zar}}\big(-,\prod_{p \in \mathbb{P}} (\mathbb{L}^j_{-/\Z})^\wedge_p\big)$, this is similarly a formal consequence of Proposition~\ref{propositionprocdhdescentforcotangentcomplex} and Lemma~\ref{lemmaprozeroimpliesprozeropcompletion}. And for $R\Gamma_{\text{Zar}}\big(-,\mathbb{L}^j_{-_{\Q}/\Q}\big)$, this is a consequence of Proposition~\ref{propositionprocdhdescentforcotangentcomplex} and the fact that the rationalisation of a zero pro system of abelian groups is zero.
	\end{proof}
	
	\subsection{Pro cdh descent for motivic cohomology}\label{subsectionprocdhdescentmotiviccohomology}
	
	\vspace{-\parindent}
	\hspace{\parindent}
	
	In this subsection, we prove pro cdh descent for the motivic complexes $\Z(i)^{\text{mot}}$ (Theo\-rem \ref{theoremprocdhdescentformotiviccohomology}). We use the fracture square \cite[Corollary~$4.31$]{bouis_motivic_2024} to decompose the proof into several steps, which ultimately all rely on Corollary~\ref{corollaryprocdhdescentforcotangentcomplexandvariants}. We start with the following rational results.
	
	\begin{proposition}[\cite{elmanto_motivic_2023}]\label{propositionprocdhdescentforderiveddeRhamchar0}
		For every integer $i \geq 0$, the presheaf $R\Gamma_{\emph{Zar}}\big(-,\widehat{\mathbb{L}\Omega}^{\geq i}_{-_{\Q}/\Q}\big)$ is a pro cdh sheaf on noetherian schemes.
	\end{proposition}
	
	\begin{proof}
		This is a part of \cite[proof of Theorem~$8.2$]{elmanto_motivic_2023}. More precisely, one uses \cite[Proposition~$4.41$]{bouis_motivic_2024} to reduce the proof to a finite number of powers of the cotangent complex relative to $\Q$, where this is Corollary~\ref{corollaryprocdhdescentforcotangentcomplexandvariants}. 
	\end{proof}
	
	The following result is a rigid-analytic variant of Proposition~\ref{propositionprocdhdescentforderiveddeRhamchar0}, where the relevant objects are defined in \cite[Section~$4.2$]{bouis_motivic_2024}.
	
	\begin{proposition}\label{propositionprocdhdescentforrigidanalyticdR}
		For every integer $i \geq 0$, the presheaf $R\Gamma_{\emph{Zar}}\big(-,\prod'_{p \in \mathbb{P}}  \underline{\widehat{\mathbb{L}\Omega}}^{\geq i}_{-_{\Q_p}/\Q_p}\big)$ is a pro cdh sheaf on noetherian schemes.
	\end{proposition}
	
	\begin{proof}
		By \cite[Remark~$4.27$]{bouis_motivic_2024}, there is a fibre sequence of presheaves
		$$R\Gamma_{\text{Zar}}\Big(-,\prod_{p \in \mathbb{P}}{}^{'}  \underline{\widehat{\mathbb{L}\Omega}}^{\geq i}_{-_{\Q_p}/\Q_p}\Big) \longrightarrow R\Gamma_{\text{Zar}}\Big(-,\prod_{p \in \mathbb{P}}{}^{'}  \underline{\widehat{\mathbb{L}\Omega}}_{-_{\Q_p}/\Q_p}\Big) \longrightarrow R\Gamma_{\text{Zar}}\Big(-,\big(\prod_{p \in \mathbb{P}} \big(\mathbb{L}\Omega^{<i}_{-/\Z}\big)^\wedge_p\big)_{\Q}\Big)$$
		on qcqs derived schemes, and in particular on noetherian schemes. By \cite[Corollary~$4.42$]{bouis_motivic_2024}, the presheaf $R\Gamma_{\text{Zar}}\big(-,\prod'_{p \in \mathbb{P}} \widehat{\underline{\mathbb{L}\Omega}}_{-_{\Q_p}/\Q_p}\big)$ is a cdh sheaf on noetherian schemes, so it is a pro cdh sheaf on noetherian schemes. The presheaf $R\Gamma_{\text{Zar}}\big(-,\big(\prod_{p \in \mathbb{P}} \big(\mathbb{L}\Omega^{<i}_{-/\Z}\big)^\wedge_p\big)_{\Q}\big)$ has a finite filtration with graded pieces given by the presheaves $R\Gamma_{\text{Zar}}\big(-,\big(\prod_{p \in \mathbb{P}} \big(\mathbb{L}^j_{-/\Z}\big)^\wedge_p\big)_{\Q}\big)$ ($0 \leq j < i$). These presheaves are pro cdh sheaves on noetherian schemes by Corollary~\ref{corollaryprocdhdescentforcotangentcomplexandvariants}, so the presheaf $R\Gamma_{\text{Zar}}\big(-,\big(\prod_{p \in \mathbb{P}} \big(\mathbb{L}\Omega^{<i}_{-/\Z}\big)^\wedge_p\big)_{\Q}\big)$ is a pro cdh sheaf on noetherian schemes. This implies that the presheaf $R\Gamma_{\text{Zar}}\big(-,\prod'_{p \in \mathbb{P}}  \underline{\widehat{\mathbb{L}\Omega}}^{\geq i}_{-_{\Q_p}/\Q_p}\big)$ is a pro cdh sheaf on noetherian schemes.
	\end{proof}
	
	\begin{proposition}\label{propositionprocdhdescentforQ(i)TC}
		For every integer $i \geq 0$, the presheaf $\Q(i)^{\emph{mot}}$ is a pro cdh sheaf on noetherian schemes.
	\end{proposition}
	
	\begin{proof}
		By \cite[Corollary~$4.67$]{bouis_motivic_2024}, there is a fibre sequence of presheaves
		$$\Q(i)^{\text{mot}}(-) \longrightarrow \Q(i)^{\text{cdh}}(-) \longrightarrow \text{cofib}\Big(R\Gamma_{\text{Zar}}\big(-,\mathbb{L}\Omega^{<i}_{-_{\Q}/\Q}\big) \longrightarrow R\Gamma_{\text{cdh}}\big(-,\Omega^{<i}_{-_{\Q}/\Q}\big)\Big)[-1]$$
		on qcqs derived schemes, and in particular on noetherian schemes. Cdh sheaves are in particular pro cdh sheaves, so it suffices to prove that the presheaf $R\Gamma_{\text{Zar}}\big(X,\mathbb{L}\Omega^{<i}_{-_{\Q}/\Q}\big)$ is a pro cdh sheaf on noetherian schemes. This presheaf has a finite filtration with graded pieces given by the presheaves $R\Gamma_{\text{Zar}}\big(-,\mathbb{L}^j_{-_{\Q}/\Q}\big)$ ($0 \leq j < i$), so the result is a consequence of Corollary~\ref{corollaryprocdhdescentforcotangentcomplexandvariants}.
		
		Alternatively, one can prove this result by using \cite[Corollary~$4.60$]{bouis_motivic_2024} and pro cdh descent for algebraic $K$-theory (\cite[Theorem~A]{kerz_algebraic_2018}).
	\end{proof}

	\begin{corollary}\label{corollaryprocdhdescentforQ(i)TC}
		For every integer $i \geq 0$, the presheaf $\Q(i)^{\emph{TC}}$ is a pro cdh sheaf on noetherian schemes.
	\end{corollary}

	\begin{proof}
		By \cite[Remark~$3.21$]{bouis_motivic_2024}, the presheaf $\Q(i)^{\text{TC}}$ is a pro cdh sheaf on noetherian schemes if and only if the presheaf $\Q(i)^{\text{mot}}$ is a pro cdh sheaf on noetherian schemes. The result is then a consequence of Proposition~\ref{propositionprocdhdescentforQ(i)TC}.
	\end{proof}
	
		By \cite[Remark~$3.21$]{bouis_motivic_2024}, the presheaf $\Q(i)^{\text{TC}}$ is a pro cdh sheaf on noetherian schemes if and only if the presheaf $\Q(i)^{\text{mot}}$ is a pro cdh sheaf on noetherian schemes. One can then prove Proposition~\ref{propositionprocdhdescentforQ(i)TC} alternatively by using the Adams decomposition (\cite[Corollary~$4.60$]{bouis_motivic_2024}) and pro cdh descent for algebraic $K$-theory (\cite[Theorem~A]{kerz_algebraic_2018}).
	
	We now turn our attention to Bhatt--Morrow--Scholze's syntomic complexes $\Z_p(i)^{\text{BMS}}$.
	
	\begin{corollary}\label{corollaryprocdhdescentforrationalisedBMSforallp}
		For every integer $i \geq 0$, the presheaf $\big(\prod_{p \in \mathbb{P}} \Z_p(i)^{\emph{BMS}}\big)_{\Q}$ is a pro cdh sheaf on noetherian schemes.
	\end{corollary}
	
	\begin{proof}
		Rationalising the cartesian square of \cite[Corollary~$4.31$]{bouis_motivic_2024} yields a cartesian square of presheaves
		$$\begin{tikzcd}
			\Q(i)^{\text{TC}}(-) \arrow{r} \arrow{d} & R\Gamma_{\text{Zar}}\Big(-,\widehat{\mathbb{L}\Omega}^{\geq i}_{-_{\Q}/\Q}\Big) \ar[d] \\
			\big(\prod_{p \in \mathbb{P}} \Z_p(i)^{\text{BMS}}(-)\big)_{\Q} \arrow{r} & R\Gamma_{\text{Zar}}\Big(-,\prod'_{p \in \mathbb{P}} \underline{\widehat{\mathbb{L}\Omega}}^{\geq i}_{-_{\Q_p}/\Q_p}\Big)
		\end{tikzcd}$$
		on qcqs derived schemes, and in particular on noetherian schemes. The other three presheaves of this cartesian square being pro cdh sheaves on noetherian schemes (Propositions~\ref{propositionprocdhdescentforderiveddeRhamchar0}, \ref{propositionprocdhdescentforrigidanalyticdR}, and \ref{propositionprocdhdescentforQ(i)TC}), the bottom left presheaf is also a pro cdh sheaf on noetherian schemes.
	\end{proof}
	
	\begin{lemma}\label{lemmaprocdhdescentformodsyntomiccohomology}
		Let $p$ be a prime number. Then for every integer $i \geq 0$, the presheaf $\F_p(i)^{\emph{BMS}}$ is a pro cdh sheaf on noetherian schemes.
	\end{lemma}
	
	\begin{proof}
		By \cite[Corollary~$5.31$]{antieau_beilinson_2020}, there exists an integer $m \geq 0$ and an equivalence of presheaves\footnote{Prismatic cohomology was first defined on $p$-complete $p$-quasisyntomic rings (\cite{bhatt_topological_2019,bhatt_prisms_2022}), and then generalised to arbitrary animated commutative rings by taking the left Kan extension from polynomial $\Z$-algebras, and imposing that it depends only on the derived $p$-completion of its input (\cite{antieau_beilinson_2020,bhatt_absolute_2022}). On noetherian rings $R$, the derived and classical $p$-completions agree, so the prismatic cohomology of $R$ is naturally identified with the prismatic cohomology of the classical $p$-completion of $R$.}
		$$\F_p(i)^{\text{BMS}}(-) \xlongrightarrow{\sim} \text{fib} \Big(\text{can}-\phi_i : (\mathcal{N}^{\geq i} \Prism_{-}\{i\}/\mathcal{N}^{\geq i+m} \Prism_{-}\{i\})/p \longrightarrow (\Prism_{-}\{i\}/\mathcal{N}^{\geq i+m} \Prism_{-}\{i\})/p\Big).$$
		In particular, it suffices to prove that for every integer $j \geq 0$, the presheaf $\mathcal{N}^j \Prism_{-}/p$ is a pro cdh sheaf on noetherian schemes. By \cite[Remark~$5.5.8$ and Example~$4.7.8$]{bhatt_absolute_2022}, there is a fibre sequence of presheaves
		$$\mathcal{N}^j \Prism_{-} \{i\}/p \longrightarrow \text{Fil}^{\text{conj}}_j \overline{\Prism}_{-/\Z_p\llbracket \widetilde{p} \rrbracket}/p \xlongrightarrow{\Theta + j} \text{Fil}^{\text{conj}}_{j-1} \overline{\Prism}_{-/\Z_p\llbracket \widetilde{p} \rrbracket}/p.$$
		The presheaves $\text{Fil}^{\text{conj}}_j \overline{\Prism}_{-/\Z_p\llbracket \widetilde{p} \rrbracket}/p$ and $\text{Fil}^{\text{conj}}_{j-1} \overline{\Prism}_{-/\Z_p\llbracket \widetilde{p} \rrbracket}/p$ have finite filtrations with graded pieces given by modulo $p$ powers of the cotangent complex. The result is then a consequence of Corollary~\ref{corollaryprocdhdescentforcotangentcomplexandvariants}.
	\end{proof}
	
	\begin{lemma}\label{lemmatorsionimpliesboundedtorsion}
		Let $A$ be an abelian group of the form $A=\prod_{p \in \mathbb{P}} A_p$, where $A_p$ is a derived $p$-complete abelian group. If $A$ is torsion, then $A$ is bounded torsion ({\it i.e.}, there exists an integer $N \geq 1$ such that $A$ is $N$-torsion).
	\end{lemma}
	
	\begin{proof}
		Assume that the abelian group $A$ is torsion. Then for every prime number $p$, the abelian group $A_p$ is torsion and derived $p$-complete, hence it is bounded $p$-power torsion by \cite[Theorem~$1.1$]{bhatt_torsion_2019}, in the sense that there exists an integer $n \geq 1$ such that $A_p[p^m]=A_p[p^n]$ for all $m \geq n$. Let $S$ be the set of prime numbers $p$ such that $A_p$ is not the zero group. Then there exists an inclusion of abelian groups $\prod_{p \in S} \F_p \subseteq A$, and, if $S$ is infinite, then $\prod_{p \in S} \F_p$ is not torsion. So $S$ is finite, and, as a finite product of bounded torsion abelian groups, the abelian group $A$ is bounded torsion.
	\end{proof}
	
	\begin{proposition}\label{propositionprocdhdescentforprodZp(i)syntomic}
		For every integer $i \geq 0$, the presheaf $\prod_{p \in \mathbb{P}} \Z_p(i)^{\emph{BMS}}$ is a pro cdh sheaf on noetherian schemes.
	\end{proposition}
	
	\begin{proof}
		Fix an abstract blowup square of noetherian schemes ($\ref{equationabstractblowupsquare}$). Let $\{C_r\}_r$ be the pro object in the derived category $\mathcal{D}(\Z)$ defined as the total fibre of the commutative square obtained by applying the presheaf $\prod_{p \in \mathbb{P}} \Z_p(i)^{\text{BMS}}$ to this abstract blowup square. We want to prove that $\{C_r\}_r$ is weakly zero. By Corollary~\ref{corollaryprocdhdescentforrationalisedBMSforallp}, its rationalisation $\{C_r \otimes_{\Z} \Q\}_r$ is weakly zero.
		
		Let $r_0 \geq 0$ and $n \in \Z$ be integers. Let $r_1 \geq r_0$ be an integer such that the map $$\text{H}^n(C_{r_1}) \otimes_{\Z} \Q \longrightarrow \text{H}^n(C_{r_0}) \otimes_{\Z} \Q$$ is the zero map. We now construct an integer $r_2 \geq r_1$ such that the map $$\text{H}^n(C_{r_2}) \longrightarrow \text{H}^n(C_{r_0})$$ is the zero map. By Lemma~\ref{lemmaprocdhdescentformodsyntomiccohomology}, and for every prime number $p$, the pro abelian group $\{\text{H}^n(C_r/p)\}_r$ is zero, which implies that the pro abelian group $\{\text{H}^n(C_r)/p\}_r$ is zero. By induction, this implies that for every integer $N \geq 1$, the pro abelian group $\{\text{H}^n(C_r)/N\}_r$ is zero. By construction, the cohomology groups $\text{H}^n(C_r)$ ($r \geq 0$) are naturally products, indexed by prime numbers $p$, of derived $p$-complete abelian groups. The kernel and cokernel of a map of derived $p$-complete abelian groups are derived $p$-complete abelian groups. So the image $A_{r_0}$ of the map $\text{H}^n(C_{r_1}) \rightarrow \text{H}^n(C_{r_0})$ is a product, indexed by prime numbers $p$, of derived $p$-complete abelian groups. This abelian group $A_{r_0}$ is also torsion by definition of the integer~$r_1$, so Lemma~\ref{lemmatorsionimpliesboundedtorsion} implies that there exists an integer $N \geq 1$ such that $A_{r_0}$ is $N$-torsion. Let~$r_2 \geq r_1$ be an integer such that the map $\text{H}^n(C_{r_2})/N \rightarrow \text{H}^n(C_{r_1})/N$ is the zero map. Then the map $$\text{H}^n(C_{r_2}) \longrightarrow \text{H}^n(C_{r_0})$$ factors as
		$$\text{H}^n(C_{r_2}) \longrightarrow \text{H}^n(C_{r_2})/N \xlongrightarrow{0} \text{H}^n(C_{r_1})/N \longrightarrow A_{r_0} \subseteq \text{H}^n(C_{r_0}),$$
		and is thus the zero map, which concludes the proof.
	\end{proof}
	
	\begin{corollary}\label{corollaryprocdhdescentforsyntomiccomplexes}
		Let $p$ be a prime number. Then for every integer $i \geq 0$, the presheaf $\Z_p(i)^{\emph{BMS}}$ is a pro cdh sheaf on noetherian schemes.
	\end{corollary}
	
	\begin{proof}
		The presheaf $\Z_p(i)^{\text{BMS}}$ is a direct summand of the presheaf $\prod_{\ell \in \mathbb{P}} \Z_{\ell}(i)^{\text{BMS}}$, so the result is a consequence of Proposition~\ref{propositionprocdhdescentforprodZp(i)syntomic}.
	\end{proof}
	
	\begin{proposition}\label{propositionprocdhdescentforZ(i)TC}
		For every integer $i \geq 0$, the presheaf $\Z(i)^{\emph{TC}}$ is a pro cdh sheaf on noetherian schemes.
	\end{proposition}
	
	\begin{proof}
		By \cite[Corollary~$4.31$]{bouis_motivic_2024}, there is a cartesian square of presheaves
		$$\begin{tikzcd}
			\Z(i)^{\text{TC}} \ar[d] \ar[r] & R\Gamma_{\text{Zar}}\big(-,\widehat{\mathbb{L}\Omega}^{\geq i}_{-_{\Q}/\Q}\big) \ar[d] \\
			\prod_{p \in \mathbb{P}} \Z_p(i)^{\text{BMS}} \ar[r] & R\Gamma_{\text{Zar}}(-,\prod'_{p \in \mathbb{P}} \underline{\widehat{\mathbb{L}\Omega}}^{\geq i}_{-_{\Q_p}/\Q_p}\big)
		\end{tikzcd}$$
		on qcqs derived schemes, and in particular on noetherian schemes. The presheaves $$R\Gamma_{\text{Zar}}\big(-,\widehat{\mathbb{L}\Omega}^{\geq i}_{-_{\Q}/\Q}\big), \text{ } R\Gamma_{\text{Zar}}\big(-,\prod_{p \in \mathbb{P}}{}^{'} \underline{\widehat{\mathbb{L}\Omega}}^{\geq i}_{-_{\Q_p}/\Q_p}\big), \text{ and } \prod_{p \in \mathbb{P}} \Z_p(i)^{\text{BMS}}$$ are pro cdh sheaves on noetherian schemes by Propositions~\ref{propositionprocdhdescentforderiveddeRhamchar0}, \ref{propositionprocdhdescentforrigidanalyticdR}, and \ref{propositionprocdhdescentforprodZp(i)syntomic} respectively. So the presheaf $\Z(i)^{\text{TC}}$ is a pro cdh sheaf on noetherian schemes.
	\end{proof}

	The following result was proved on noetherian schemes over a field by Elmanto--Morrow \cite{elmanto_motivic_2023}.
	
	\begin{theorem}[Pro cdh descent]\label{theoremprocdhdescentformotiviccohomology}
		For every integer $i \geq 0$, the motivic complex $\Z(i)^{\emph{mot}}$ is a pro cdh sheaf on noetherian schemes.
	\end{theorem}
	
	\begin{proof}
		By \cite[Remark~$3.21$]{bouis_motivic_2024}, there is a cartesian square of presheaves
		$$\begin{tikzcd}
			\Z(i)^{\text{mot}} \ar[d] \ar[r] & \Z(i)^{\text{TC}} \ar[d] \\
			\Z(i)^{\text{cdh}} \ar[r] & L_{\text{cdh}} \Z(i)^{\text{TC}}
		\end{tikzcd}$$
		on qcqs schemes, and in particular on noetherian schemes. The presheaf $\Z(i)^{\text{TC}}$ is a pro cdh sheaf on noetherian schemes by Proposition~\ref{propositionprocdhdescentforZ(i)TC}. The presheaves $\Z(i)^{\text{cdh}}$ and $L_{\text{cdh}} \Z(i)^{\text{TC}}$ are cdh sheaves on noetherian schemes by construction, hence pro cdh sheaves on noetherian schemes. So the presheaf $\Z(i)^{\text{mot}}$ is a pro cdh sheaf.
	\end{proof}

	\begin{remark}[Pro cdh descent for algebraic $K$-theory]\label{remarkprocdhdescentforKtheory}
		The arguments to prove Theorem~\ref{theoremprocdhdescentformotiviccohomology} can be adapted to give a new proof of the pro cdh descent for algebraic $K$-theory of Kerz--Strunk--Tamme \cite{kerz_algebraic_2018}. More precisely, by \cite[Theorem~$1.1$]{bouis_motivic_2024}, pro cdh descent for algebraic $K$-theory is equivalent to pro cdh descent for TC. By \cite[Corollary~$4.63$]{bouis_motivic_2024}, the result rationally reduces to the pro cdh descent for HC, which is proved by Morrow (\cite[Theorem~$0.2$]{morrow_pro_2016}). The result mod $p$ is similar to that of Lemma~\ref{lemmaprocdhdescentformodsyntomiccohomology}, where the Nygaard filtration and the relative prismatic cohomology are replaced by the Tate filtration and by relative THH; the pro cdh descent for relative THH then reduces to the pro cdh descent for powers of the cotangent complex by \cite[Section~$5.2$]{antieau_beilinson_2020}. Following \cite[Section~$4.2$]{bouis_motivic_2024}, there is a natural cartesian square
		$$\begin{tikzcd}
			\text{TC}(-) \ar[r] \ar[d] & \text{HC}^-(-_{\Q}/\Q) \ar[d] \\
			\prod_{p \in \mathbb{P}} \text{TC}(-;\Z_p) \ar[r] & \Big(\prod'_{p \in \mathbb{P}} \text{HH}(-;\Q_p)\Big)^{h\text{S}^1}.
		\end{tikzcd}$$
		Using the cdh descent for the presheaves $\text{HP}(-_{\Q}/\Q)$ (\cite{land_k-theory_2019}) and $\big(\prod'_{p \in \mathbb{P}} \text{HH}(-;\Q_p)\big)^{t\text{S}^1}$ (\cite[Corollary~$4.42$]{bouis_motivic_2024}), the pro cdh descent of the two right terms reduces to the pro cdh descent for HC. The integral statement is then similarly a consequence of Lemma~\ref{lemmatorsionimpliesboundedtorsion}.
	\end{remark}
	
	\subsection{Motivic Weibel vanishing}
	
	\vspace{-\parindent}
	\hspace{\parindent}
	
	In this subsection, we prove Theorem~\ref{theoremmotivicWeibelvanishing}, which is a motivic refinement of Weibel's vanishing conjecture on negative $K$-groups (\cite[Theorem~B\,$(i)$]{kerz_algebraic_2018}).
	
	\begin{lemma}\label{lemmamotiviccohomologyofhenselianvaluationringisindegreesatmosti}
		Let $V$ be a henselian valuation ring. Then for every integer $i \geq 0$, the motivic complex $\Z(i)^{\emph{mot}}(V) \in \mathcal{D}(\Z)$ is in degrees at most $i$.
	\end{lemma}
	
	\begin{proof}
		Henselian valuation rings are local rings for the cdh topology, so the natural maps
		$$\Z(i)^{\text{mot}}(V) \longrightarrow \Z(i)^{\text{cdh}}(V) \longleftarrow \Z(i)^{\text{lisse}}(V)$$
		are equivalences in the derived category $\mathcal{D}(\Z)$ (\cite[Remark~$3.24$ and Definition~$3.10$]{bouis_motivic_2024}).
	\end{proof}
	
	\begin{lemma}\label{lemmanilfibreofmotiviccohomologyisindegreesatmosti}
		Let $A$ be a local ring, and $I$ be a nil ideal of $A$. Then for every integer $i \geq 0$, the fibre of the natural map
		$$\Z(i)^{\emph{mot}}(A) \longrightarrow \Z(i)^{\emph{mot}}(A/I)$$
		is in degrees at most $i$.
	\end{lemma}
	
	\begin{proof}
		We first prove the result rationally, and modulo $p$ for every prime number $p$. Any finitary cdh sheaf is invariant under nil extensions. By \cite[Corollary~$4.67$]{bouis_motivic_2024}, the result after rationalisation is thus equivalent to the fact that the fibre of the natural map
		$$\mathbb{L}\Omega^{<i}_{(A_{\Q})/\Q}[-1] \longrightarrow \mathbb{L}\Omega^{<i}_{((A/I)_{\Q})/\Q}[-1]$$
		is in degrees at most $i$. Both terms of this map are in degrees at most $i$. In degree $i$, this map is given by the natural map
		$$\Omega^{i-1}_{(A_{\Q})/\Q} \longrightarrow \Omega^{i-1}_{((A/I)_{\Q})/\Q},$$
		which is surjective as the $\Q$-algebra $(A/I)_{\Q}$ is a quotient of the $\Q$-algebra $A_{\Q}$. Let $p$ be a prime number. By \cite[Corollary~$3.26$]{bouis_motivic_2024}, the result modulo $p$ is equivalent to the fact that the fibre of the natural map
		$$\F_p(i)^{\text{BMS}}(A) \longrightarrow \F_p(i)^{\text{BMS}}(A/I)$$
		is in degrees at most $i$. The pair $(A,I)$ is henselian, so this is a consequence of \cite[Theorem~$2.27$]{bouis_motivic_2024}. 
		
		By the previous rational statement, the fibre $F \in \mathcal{D}(\Z)$ of the natural map $$\Z(i)^{\text{mot}}(A) \longrightarrow \Z(i)^{\text{mot}}(A/I)$$ has torsion cohomology groups in degrees at least $i+1$. By the short exact sequence of abelian groups
		$$0 \longrightarrow \text{H}^j(F)/p \longrightarrow \text{H}^j(F/p) \longrightarrow \text{H}^{j+1}(F)[p] \longrightarrow 0$$
		for every prime number $p$ and every integer $j \geq i+1$,
		the previous torsion statement implies that these cohomology groups are also torsionfree, hence zero, in degrees at least $i+2$. It then remains to prove that the abelian group $\text{H}^{i+1}(F)$ is zero. By Corollary~\ref{corollaryHilbert90} and its proof, the abelian group $\text{H}^{i+1}_{\text{mot}}(A,\Z(i))$ is torsionfree, so it suffices to prove that the natural map of abelian groups $\text{H}^i_{\text{mot}}(A,\Z(i)) \rightarrow \text{H}^i_{\text{mot}}(A/I,\Z(i))$ is surjective. Let $P$ be a local ind-smooth $\Z$-algebra with a surjective map $P \rightarrow A$. By Theorem~\ref{theoremmotiviccohomologyisleftKanextendedonlocalringsinsmalldegrees} (see also Lemma~\ref{lemmasymbolmapfactorsthoughimprovedMilnorKgroup} for a related argument), and because $P \rightarrow A/I$ is also a surjection from a local ind-smooth $\Z$-algebra, the composite map of abelian groups
		$$\text{H}^i_{\text{mot}}(P,\Z(i)) \longrightarrow \text{H}^i_{\text{mot}}(A,\Z(i)) \longrightarrow \text{H}^i_{\text{mot}}(A/I,\Z(i))$$
		is surjective, so the right map is surjective, as desired.
	\end{proof}
	
	\begin{theorem}[Motivic Weibel vanishing]\label{theoremmotivicWeibelvanishing}
		Let $d \geq 0$ be an integer, and $X$ be a noetherian scheme of dimension at most $d$. Then for every integer $i \geq 0$, the motivic complex $\Z(i)^{\emph{mot}}(X) \in \mathcal{D}(\Z)$ is in degrees at most $i+d$.  
	\end{theorem}
	
	\begin{proof}
		The presheaf $\Z(i)^{\text{mot}} : \text{Sch}^{\text{qcqs,op}} \rightarrow \mathcal{D}(\Z)$ satisfies the following properties:
		\begin{enumerate}
			\item it is finitary (\cite[Corollary~$4.59$]{bouis_motivic_2024});
			\item it satisfies pro cdh descent on noetherian schemes (Theorem~\ref{theoremprocdhdescentformotiviccohomology});
			\item for every henselian valuation ring $V$, the complex $\Z(i)^{\text{mot}}(V)$ is in degrees at most~$i$ (Lemma~\ref{lemmamotiviccohomologyofhenselianvaluationringisindegreesatmosti});
			\item for every noetherian local ring $A$ and every nilpotent ideal $I$ of $A$, the fibre of the natural map
			$\Z(i)^{\text{mot}}(A) \rightarrow \Z(i)^{\text{mot}}(A/I)$
			is in degrees at most $i$ (Lemma~\ref{lemmanilfibreofmotiviccohomologyisindegreesatmosti}). 
		\end{enumerate}
		By \cite[Definition~$3.12$ and Lemma~$4.57$]{bouis_motivic_2024}, the presheaf $\Z(i)^{\text{cdh}} : \text{Sch}^{\text{qcqs,op}} \rightarrow \mathcal{D}(\Z)$ is a finitary cdh sheaf which is in degrees at most $i$ on henselian valuation rings, hence it also satisfies the previous properties. 
		
		By \cite[Proposition~$8.10$]{elmanto_motivic_2023} applied to the presheaf $\text{fib}\big(\Z(i)^{\text{mot}} \rightarrow \Z(i)^{\text{cdh}}\big)[i]$, this implies that for every noetherian scheme $X$ of dimension at most $d$, the complex $$\text{fib}\big(\Z(i)^{\text{mot}}(X) \longrightarrow \Z(i)^{\text{cdh}}(X)\big)$$ is in degrees at most $i+d$. The complex $\Z(i)^{\text{cdh}}(X)$ is also in degrees at most $i+d$ (\cite[Theorem~$7.12$]{bachmann_A^1-invariant_2024}, see also \cite[Theorem~$3.13\,(1)$]{bouis_motivic_2024}), so the complex $\Z(i)^{\text{mot}}(X)$ is in degrees at most $i+d$.
	\end{proof}
	
	\begin{remark}[Relation to Weibel's $K$-theoretic vanishing conjecture]\label{remarkrelationWeibelKtheoreticconjecture}
		Let $X$ be a noetherian scheme of dimension at most $d$. Theorem~\ref{theoremmotivicWeibelvanishing} states that the Atiyah--Hirzebruch spectral sequence
		$$E_2^{i,j} = \text{H}^{i-j}_{\text{mot}}(X,\Z(-j)) \Longrightarrow \text{K}_{-i-j}(X)$$
		is supported in the left half plane $x \leq d$: see the following representation of the $E_2$ page, where $\text{H}^j(i)$ denotes the motivic cohomology group $\text{H}^j_{\text{mot}}(X,\Z(i))$.
		\medskip
		
		\begin{equation*}
			\begin{tikzcd}[sep=tiny]
				\cdots & 0 & 0 & 0 & 0 & \cdots & 0 & 0 & 0 & 0 \\
				\cdots & 0 & 0 & \text{H}^0(0) & \text{H}^1(0) & \cdots & \text{H}^{d-2}(0) \ar[rrd] & \text{H}^{d-1}(0) & \text{H}^d(0) & 0 \\
				\cdots & 0 & \text{H}^0(1) \ar[rrd] & \text{H}^1(1) & \text{H}^2(1) & \cdots & \text{H}^{d-1}(1) \ar[rrd] & \text{H}^d(1) & \text{H}^{d+1}(1) & 0 \\
				\cdots & \text{H}^0(2) & \text{H}^1(2) & \text{H}^2(2) & \text{H}^3(2) & \cdots & \text{H}^d(2) & \text{H}^{d+1}(2) & \text{H}^{d+2}(2) & 0 \\
				& \vdots & \vdots & \vdots & \vdots & \vdots & \vdots & \vdots & \vdots & \vdots
			\end{tikzcd}
		\end{equation*}
		\medskip
		
		In particular, the negative $K$-groups $\text{K}_{-i-j}(X)$ vanish for $-i-j<-d$ (this is Weibel's vanishing conjecture on algebraic $K$-theory), and there is a natural edge map isomorphism
		$$\text{K}_{-d}(X) \cong \text{H}^d_{\text{mot}}(X,\Z(0))$$
		of abelian groups. Using the description of weight zero motivic cohomology (\cite[Example~$4.68$]{bouis_motivic_2024}), the latter result recovers the known description of $\text{K}_{-d}(X)$ (\cite[Corollary~D]{kerz_algebraic_2018}). Note that Theorem~\ref{theoremmotivicWeibelvanishing} is however not a new proof of these results of Kerz--Strunk--Tamme, as our Atiyah--Hirzebruch spectral sequence relating motivic cohomology and algebraic $K$-theory relies on \cite[Theorem~$1.1$]{bouis_motivic_2024}, which itself relies on the results in \cite{kerz_algebraic_2018}.  
	\end{remark}
	
	\begin{remark}\label{remarkdescriptionofK_-d}
		Let $X$ be a noetherian scheme of dimension at most $d$. Then for every integer~$i \geq 0$, the proof of Theorem~\ref{theoremmotivicWeibelvanishing} also implies that the natural map $$\Z(i)^{\text{mot}}(X) \longrightarrow \Z(i)^{\text{cdh}}(X)$$ is surjective on $\text{H}^{i+d}$. For $i=0$, this map is even an isomorphism on $\text{H}^d$ (actually on all cohomology groups, by \cite[Example~$4.68$]{bouis_motivic_2024}), thus recovering Weibel's conjecture that the natural map $\text{K}_{-d}(X) \rightarrow \text{KH}_{-d}(X)$ is an isomorphism \cite{weibel_K-theory_1980,kerz_algebraic_2018}.
	\end{remark}
	
	The following result is a description of the group $\text{K}_{-d+1}$, similar to the description of the group $\text{K}_{-d}$ predicted by Weibel (Remark~\ref{remarkrelationWeibelKtheoreticconjecture}).
	
	\begin{corollary}\label{corollaryweibelupgraded}
		Let $d \geq 0$ be an integer, and $X$ be a noetherian scheme of dimension at most $d$. Then there is a natural exact sequence
		$$\emph{H}^{d-2}_{\emph{cdh}}(X,\Z) \xlongrightarrow{\delta} \emph{H}^{d+1}_{\emph{mot}}(X,\Z(1)) \longrightarrow \emph{K}_{-d+1}(X) \longrightarrow \emph{H}^{d-1}_{\emph{cdh}}(X,\Z) \longrightarrow 0$$
		of abelian groups, where $\delta$ is the differential map coming from the $E_2$-page of the Atiyah--Hirzebruch spectral sequence (\cite[Corollary~$4.55$]{bouis_motivic_2024}). Moreover, for every integer $m \geq 2$, if $m$ is invertible in~$X$, then the image of the map $(m-1)\delta$ is $m$-power torsion.
	\end{corollary}
	
	\begin{proof}
		The motivic complex $\Z(i) ^{\text{mot}}(X)$ is zero for $i<0$ (\cite[Corollary~$4.47$]{bouis_motivic_2024}), and is naturally identified with the complex $R\Gamma_{\text{cdh}}(X,\Z)$ for $i=0$ (\cite[Example~$4.68$]{bouis_motivic_2024}). The first statement is then a consequence of the Atiyah--Hirzebruch spectral sequence (\cite[Corollary~$4.55$]{bouis_motivic_2024}) and of motivic Weibel's vanishing (Theorem~\ref{theoremmotivicWeibelvanishing}). The second statement is a consequence of the compatibility of the map $\delta$ with the Adams operation $\psi^m$ (\cite[Construction~$4.9$]{bouis_motivic_2024}). More precisely, \cite[Corollary~$4.10$]{bouis_motivic_2024} implies that the induced map $$\delta : \text{H}^{d-2}_{\text{cdh}}(X,\Z)[\tfrac{1}{m}] \longrightarrow \text{H}^{d+1}_{\text{mot}}(X,\Z(1))[\tfrac{1}{m}]$$
		satisfies the equation $\delta = m\delta$, {\it i.e.}, $(m-1)\delta=0$, which implies the desired result.
	\end{proof}

	\subsection{Comparison to pro cdh motivic cohomology}
	
	\vspace{-\parindent}
	\hspace{\parindent}
	
	In this subsection, we compare the motivic complexes $\Z(i)^{\text{mot}}$ to Kelly--Saito's pro cdh motivic complexes $\Z(i)^{\text{procdh}}$ (Theorem~\ref{theoremcomparisonprocdhmotivic}). In equicharacteristic, this is \cite[Corollary~$1.11$]{kelly_procdh_2024} (see also \cite[Theorem~$1.15$]{elmanto_motivic_2023}). Our proof is structurally the same, although finitariness and pro cdh descent in mixed characteristic rely on the main results of \cite[Section~$4$]{bouis_motivic_2024}, and our proof of the comparison to lisse motivic cohomology is different in mixed characteristic (see comment before Corollary~\ref{corollarylissemotivicmaincomparisontheorem}).
	
	\begin{lemma}\label{lemmanilextensionhenselianvaluationringlisse=mot}
		Let $R$ be a nil extension of a henselian valuation ring, {\it i.e.}, a commutative ring $R$ whose quotient $R/I$ by its ideal of nilpotent elements $I$ is a henselian valuation ring. Then for every integer $i \geq 0$, the lisse-motivic comparison map
		$$\Z(i)^{\emph{lisse}}(R) \longrightarrow \Z(i)^{\emph{mot}}(R)$$
		is an equivalence in the derived category $\mathcal{D}(\Z)$.
	\end{lemma}
	
	\begin{proof}
		By Corollary~\ref{corollarylissemotivicmaincomparisontheorem}, it suffices to prove that the complex $\Z(i)^{\text{mot}}(R) \in \mathcal{D}(\Z)$ is in degrees at most~$i$. Let $I$ be the ideal of nilpotent elements of the commutative ring $R$. Using the natural fibre sequence
		$$\text{fib}\big(\Z(i)^{\text{mot}}(R) \longrightarrow \Z(i)^{\text{mot}}(R/I)\big) \longrightarrow \Z(i)^{\text{mot}}(R) \longrightarrow \Z(i)^{\text{mot}}(R/I)$$
		in the derived category $\mathcal{D}(\Z)$, the result is then a consequence of Lemmas~\ref{lemmamotiviccohomologyofhenselianvaluationringisindegreesatmosti} and~\ref{lemmanilfibreofmotiviccohomologyisindegreesatmosti}.
	\end{proof}
	
	\begin{theorem}[Comparison to pro cdh motivic cohomology]\label{theoremcomparisonprocdhmotivic}
		Let $X$ be a noetherian scheme. Then for every integer $i \geq 0$, the lisse-motivic comparison map induces a natural equivalence
		$$\Z(i)^{\emph{procdh}}(X) \xlongrightarrow{\sim} \Z(i)^{\emph{mot}}(X)$$
		in the derived category $\mathcal{D}(\Z)$.
	\end{theorem}
	
	\begin{proof}
		The presheaf $\Z(i)^{\text{mot}} : \text{Sch}^{\text{qcqs,op}} \rightarrow \mathcal{D}(\Z)$ satisfies the following properties:
		\begin{enumerate}
			\item it is finitary (\cite[Corollary~$4.59$]{bouis_motivic_2024});
			\item it satisfies pro cdh descent on noetherian schemes (Theorem~\ref{theoremprocdhdescentformotiviccohomology});
			\item for every pro cdh local ring $R$, the lisse-motivic comparison map
			$$\Z(i)^{\text{lisse}}(R) \longrightarrow \Z(i)^{\text{mot}}(R)$$
			is an equivalence in the derived category $\mathcal{D}(\Z)$ (Lemma~\ref{lemmanilextensionhenselianvaluationringlisse=mot} and \cite[Proposition~$1.7$]{kelly_procdh_2024}).
		\end{enumerate}
		By \cite[Theorem~$9.7$]{kelly_procdh_2024},\footnote{This theorem is stated for schemes over a field, but the proof works over any noetherian commutative ring.} this implies that for every noetherian scheme $X$, the lisse-motivic comparison map induces a natural equivalence
		$$\Z(i)^{\text{procdh}}(X) \xlongrightarrow{\sim} \Z(i)^{\text{mot}}(X)$$
		in the derived category $\mathcal{D}(\Z)$.
	\end{proof}

	\section{The projective bundle formula}\label{sectionpbf}
	
	\vspace{-\parindent}
	\hspace{\parindent}
	
	In this section, we prove that the motivic complexes $\Z(i)^{\text{mot}}$ satisfy the projective bundle formula (Theorem~\ref{theoremprojectivebundleformula}) and regular blowup excision (Theorem~\ref{theoremregularblowupformula}). This implies in particular that the presheaves $\Z(i)^{\text{mot}}$ fit within the recent theory of non-$\mathbb{A}^1$-invariant motives of Annala--Iwasa \cite{annala_motivic_2023} and Annala--Hoyois--Iwasa \cite{annala_algebraic_2025,annala_atiyah_2024}.
	
	\subsection{First Chern classes}\label{subsectionfirstChernclasses}
	
	\vspace{-\parindent}
	\hspace{\parindent}
	
	In this subsection, we construct the motivic first Chern class (Definition~\ref{definitionmotivicfirstChernclass}) in order to formulate the projective bundle formula (Theorem~\ref{theoremprojectivebundleformula}).
	
	\begin{definition}[Motivic first Chern class]\label{definitionmotivicfirstChernclass}
		Let $X$ be a qcqs derived scheme. The {\it motivic first Chern class} is the natural map
		$$c_1^{\text{mot}} : R\Gamma_{\text{Nis}}(X,\mathbb{G}_m)[-1] \longrightarrow \Z(1)^{\text{mot}}(X),$$
		in the derived category $\mathcal{D}(\Z)$, defined as the Nisnevich sheafification of the natural map of presheaves 
		$$\big(\tau^{\leq 1} R\Gamma_{\text{Zar}}(-,\mathbb{G}_m)\big)[-1] \longrightarrow \Z(1)^{\text{mot}}(-)$$
		induced by Definition~\ref{definitionlissemotiviccomparisonmap} and \cite[Example~$3.9$]{bouis_motivic_2024}. We also denote by
		$$c_1^{\text{mot}} : \text{Pic}(X) \longrightarrow \text{H}^2_{\text{mot}}(X,\Z(1))$$
		the map induced on $\text{H}^2$ (Example~\ref{exampleweightonemotiviccohomology}).
	\end{definition}
	
	\begin{remark}\label{remarkunicityfirstChernclassfromsmoothcase}
		The motivic first Chern class of Definition~\ref{definitionmotivicfirstChernclass} is uniquely determined by its naturality, and the fact that it is given by the map of \cite[Definition~$3.23$]{bouis_motivic_2024} on smooth $\Z$-schemes.
	\end{remark}
	
	For every qcqs scheme $X$, the line bundle $\mathcal{O}(1) \in \text{Pic}(\mathbb{P}^1_X)$, via the multiplicative structure of the motivic complexes $\Z(i)^{\text{mot}}$, induces, for every integer $i \in \Z$, a natural map
	
	$$c_1^{\text{mot}}(\mathcal{O}(1)) : \Z(i-1)^{\text{mot}}(\mathbb{P}^1_X)[-2] \longrightarrow \Z(i)^{\text{mot}}(\mathbb{P}^1_X)$$
	in the derived category $\mathcal{D}(\Z)$. If $\pi : \mathbb{P}^1_X \rightarrow X$ is the canonical projection map, this in turn induces a natural map
	\begin{equation}\label{equationmapforP^1bundleformule}
		\pi^\ast \oplus c_1^{\text{mot}}(\mathcal{O}(1))\pi^\ast : \Z(i)^{\text{mot}}(X) \oplus \Z(i-1)^{\text{mot}}(X)[-2] \longrightarrow \Z(i)^{\text{mot}}(\mathbb{P}^1_X)
	\end{equation}
	in the derived category $\mathcal{D}(\Z)$. The aim of the following section is to prove that this map is an equivalence (Theorem~\ref{theoremP^1bundleformulaformotiviccohomology}). To prove such an equivalence, we will need compatibilities with other first Chern classes.
	
	\begin{construction}[$\mathbb{P}^1$-bundle formula for additive invariants]
		Following \cite[Section~$5.1$]{elmanto_motivic_2023}, every additive invariant $E$ of $\Z$-linear $\infty$-categories in the sense of \cite[Definition~$5.11$]{hoyois_higher_2017} admits a natural first Chern class, inducing a natural map of spectra
		$$\pi^\ast \oplus (1-c_1)(\mathcal{O}(-1))\pi^\ast : E(X) \oplus E(X) \longrightarrow E(\mathbb{P}^1_X).$$
		For every additive invariant of $\Z$-linear $\infty$-categories, this map is an equivalence (\cite[Lemma~$5.6$]{elmanto_motivic_2023}). 
	\end{construction}
	
	\begin{remark}[Compatibility with filtrations]\label{remarkfilteredadditiveinvariants}
		If the additive invariant $E$ of $\Z$-linear $\infty$-categories, seen as a presheaf of spectra on qcqs schemes, admits a multiplicative filtered refinement $\text{Fil}^\star E$ which is a multiplicative filtered module over the lisse motivic filtration $\text{Fil}^\star_{\text{lisse}} \text{K}^{\text{conn}}$ (\cite[Definition~$3.7$]{bouis_motivic_2024}),\footnote{The lisse motivic filtration $\text{Fil}^\star_{\text{lisse}} \text{K}^{\text{conn}}$ is usually defined only on affine schemes. Here the argument works if $\text{Fil}^\star E$ is a Zariski sheaf of filtered spectra on qcqs schemes, and if its restriction to affine schemes is a multiplicative filtered module over the lisse motivic filtration $\text{Fil}^\star_{\text{lisse}} \text{K}^{\text{conn}}$.} then this map has a natural filtered refinement
		$$\pi^\ast \oplus (1-c_1)(\mathcal{O}(-1)) \pi^\ast : \text{Fil}^\star E(X) \oplus \text{Fil}^{\star - 1} E(X) \longrightarrow \text{Fil}^\star E(\mathbb{P}^1_X)$$
		by \cite[Construction~$5.11$ and Lemma~$5.12$]{elmanto_motivic_2023}. If $E$ is algebraic $K$-theory, equipped with the motivic filtration $\text{Fil}^\star_{\text{mot}} \text{K}$ (\cite[Definition~$3.19$]{bouis_motivic_2024}), the argument of \cite[Lemma~$5.12$]{elmanto_motivic_2023} and Remark~\ref{remarkunicityfirstChernclassfromsmoothcase} imply that this map recovers, up to a shift, the map $(\ref{equationmapforP^1bundleformule})$ on graded pieces.
	\end{remark}
	
	\begin{example}[Compatibility with cdh-local motivic cohomology]
		The multiplicative cdh-local filtration $\text{Fil}^\star_{\text{cdh}} \text{KH}$ (\cite[Definition~$3.10$]{bouis_motivic_2024}) is naturally a module over the multiplicative filtration $\text{Fil}^\star_{\text{mot}} \text{K}$ ({\it e.g.},~because cdh sheafification preserves multiplicative structures), so the first Chern class for the cdh\nobreakdash-local motivic complexes of \cite{bachmann_A^1-invariant_2024} is compatible with the motivic first Chern class of Definition~\ref{definitionmotivicfirstChernclass}.
	\end{example}
	
	\begin{example}[Compatibility with syntomic cohomology]\label{exampleSelmerKtheory}
		Let $X$ be a qcqs scheme, and $p$ be a prime number. As explained in \cite[Remark~$5.25$]{bachmann_A^1-invariant_2024}, the syntomic first Chern class of \cite[Section~$7$]{bhatt_absolute_2022} is compatible with the motivic first Chern class of Definition~\ref{definitionmotivicfirstChernclass} via the motivic-syntomic comparison map (\cite[Construction~$5.8$]{bouis_motivic_2024}). Note here that the motivic-syntomic comparison map can be seen as the map induced on graded pieces from a multiplicative map of filtered spectra
		$$\text{Fil}^{\star}_{\text{mot}} \text{K}(X;\Z_p) \longrightarrow \text{Fil}^\star_{\text{mot}} \text{K}^{\text{Sel}}(X;\Z_p),$$
		where the target is the filtration on $p$-completed Selmer $K$-theory (\cite[Remark~$8.4.3$]{bhatt_absolute_2022}). These motivic and syntomic first Chern classes then coincide with the first Chern classes coming from the additive invariants $\text{K}(-;\Z_p)$ and $\text{K}^{\text{Sel}}(-;\Z_p)$ (Remark~\ref{remarkfilteredadditiveinvariants}).
	\end{example}

	\subsection{\texorpdfstring{$\mathbb{P}^1$}{TEXT}-bundle formula}
	
	\vspace{-\parindent}
	\hspace{\parindent}
	
	In this subsection, we prove that the motivic complexes $\Z(i)^{\text{mot}}$ satisfy the $\mathbb{P}^1$-bundle formula on qcqs schemes (Theorem~\ref{theoremP^1bundleformulaformotiviccohomology}). Note that the $\mathbb{P}^1$-bundle formula is unknown for the cdh-local motivic complexes $\Z(i)^{\text{cdh}}$ on general qcqs schemes.\footnote{More precisely, Bachmann--Elmanto--Morrow proved in \cite{bachmann_A^1-invariant_2024} that if a qcqs scheme $X$ satisfies that every valuation ring $V$ with a map $\text{Spec}(V) \rightarrow X$ is $F$-smooth, then the cdh-local motivic complexes $\Z(i)^{\text{cdh}}$ satisfy the $\mathbb{P}^1$-bundle formula at $X$. In particular, the $\mathbb{P}^1$-bundle formula is known for these complexes over a field, and over a mixed characteristic perfectoid valuation ring by the results of \cite{bouis_cartier_2023}, but not on general qcqs schemes.} The cartesian square of \cite[Remark~$3.21$]{bouis_motivic_2024} thus cannot be used directly to prove the $\mathbb{P}^1$-bundle formula for the motivic complexes $\Z(i)^{\text{mot}}$, as was done by Elmanto--Morrow in equicharacteristic (\cite[Section~$5$]{elmanto_motivic_2023}). Instead, we use in a crucial way our main result on $p$-adic motivic cohomology (\cite[Theorem~$5.10$]{bouis_motivic_2024}), and a degeneration argument using Selmer $K$-theory.
	
	\begin{theorem}[$\mathbb{P}^1$-bundle formula]\label{theoremP^1bundleformulaformotiviccohomology}
		Let $X$ be a qcqs scheme, and $\pi : \mathbb{P}^1_X \rightarrow X$ be the canonical projection map. Then for every integer $i \in \Z$, the natural map
		$$\pi^\ast \oplus c_1^{\emph{mot}}(\mathcal{O}(1)) \pi^\ast : \Z(i)^{\emph{mot}}(X) \oplus \Z(i-1)^{\emph{mot}}(X)[-2] \longrightarrow \Z(i)^{\emph{mot}}(\mathbb{P}^1_X)$$
		is an equivalence in the derived category $\mathcal{D}(\Z)$.
	\end{theorem}
	
	Using \cite[Theorem~$5.10$]{bouis_motivic_2024}, the proof of Theorem~\ref{theoremP^1bundleformulaformotiviccohomology} will reduce to the proof of a similar equivalence for the cdh sheaves $\big(L_{\text{cdh}} \tau^{>i} \F_p(i)^{\text{syn}}\big)(-)$ (Proposition~\ref{propositionP^1bundleformulaforcdhlocalhighdegreesyntomiccohomology}). Most of this section is devoted to the study of these cdh sheaves.
	
	\begin{lemma}\label{lemmaGfibresequencecdhhighdegreesyntomiccohomologywithétaleandBMS}
		Let $p$ be a prime number. Then for any integers $i,j \geq 0$ and $k \geq 1$, the natural sequence
		$$L_{\emph{Nis}} \tau^{>j} R\Gamma_{\emph{ét}}(-,j_!\mu_{p^k}^{\otimes i}) \longrightarrow \big(L_{\emph{Nis}} \tau^{>j} \Z/p^k(i)^{\emph{syn}}\big)(-) \longrightarrow \big(L_{\emph{Nis}} \tau^{>j} \Z/p^k(i)^{\emph{BMS}}\big)(-)$$
		is a fibre sequence of $\mathcal{D}(\Z/p^k)$-valued presheaves on qcqs schemes.
	\end{lemma}
	
	\begin{proof}
		The three presheaves are finitary Nisnevich sheaves, so it suffices to prove the result on henselian local rings (\cite[Corollary~$3.27$ and Example~$4.31$]{clausen_hyperdescent_2021}). Let $A$ be a henselian local ring. By \cite[Remark~$8.4.4$]{bhatt_absolute_2022}, there is a natural fibre sequence
		$$R\Gamma_{\text{ét}}(A,j_!\mu_{p^k}^{\otimes i}) \longrightarrow \Z/p^k(i)^{\text{syn}}(A) \longrightarrow \Z/p^k(i)^{\text{BMS}}(A)$$
		in the derived category $\mathcal{D}(\Z/p^k)$, so it suffices to prove that the natural map
		$$\Z/p^k(i)^{\text{syn}}(A) \longrightarrow \Z/p^k(i)^{\text{BMS}}(A)$$
		is surjective in degree $j$. If $p$ is invertible in the henselian valuation ring $A$, the target of this map is zero. If $p$ is not invertible in $A$, then the valuation ring $A$ is $p$-henselian, and this map is an equivalence (\cite[Notation~$5.7$]{bouis_motivic_2024}).
	\end{proof}

	\begin{lemma}\label{lemmaFweirdétalecohomologygroupisconcentratedincohomologicaldegreei+1}
		Let $p$ be a prime number, and $V$ be a rank one henselian valuation ring of mixed characteristic $(0,p)$. Then for any integers $i \geq 0$ and $k \geq 1$, the complex $$\big(L_{\emph{cdh}} \tau^{>i} R\Gamma_{\emph{ét}}(-,j_!\mu_{p^k}^{\otimes i})\big)(\mathbb{P}^1_V) \in \mathcal{D}(\Z/p^k)$$ is concentrated in degree $i+1$.\footnote{We will prove, at the end of Proposition~\ref{propositionCP1bundleformulaformixedcharvaluationringrank1}, that this complex is actually zero.} 
	\end{lemma}
	
	\begin{proof}
		The presheaf $L_{\text{cdh}} \tau^{>i} R\Gamma_{\text{ét}}(-,j_!\mu_{p^k}^{\otimes i})$ is the cdh sheafification of a presheaf taking values in degrees at least $i+1$, so it takes values in degrees at least $i+1$. To prove that the complex
		$$\big(L_{\text{cdh}} \tau^{>i} R\Gamma_{\text{ét}}(-,j_!\mu_{p^k}^{\otimes i})\big)(\mathbb{P}^1_V) \in \mathcal{D}(\Z/p^k)$$
		is in degrees at most $i+1$, consider the fibre sequence
		$$\big(L_{\text{cdh}} \tau^{\leq i} R\Gamma_{\text{ét}}(-,j_!\mu_{p^k}^{\otimes i})\big)(\mathbb{P}^1_V) \longrightarrow R\Gamma_{\text{ét}}(\mathbb{P}^1_V,j_!\mu_{p^k}^{\otimes i}) \longrightarrow \big(L_{\text{cdh}} \tau^{>i} R\Gamma_{\text{ét}}(-,j_!\mu_{p^k}^{\otimes i})\big)(\mathbb{P}^1_V)$$
		in the derived category $\mathcal{D}(\Z/p^k)$, which is a consequence of arc descent for the presheaf $R\Gamma_{\text{ét}}(-,j_!\mu_{p^k}^{\otimes i})$ (\cite[Theorem~$1.8$]{bhatt_arc-topology_2021}). The scheme $\mathbb{P}^1_V$ has valuative dimension two, so the complex
		$$\big(L_{\text{cdh}} \tau^{\leq i} R\Gamma_{\text{ét}}(-,j_!\mu_{p^k}^{\otimes i})\big)(\mathbb{P}^1_V) \in \mathcal{D}(\Z/p^k)$$
		is in degrees at most $i+2$ (\cite[Theorem~$2.4.15$]{elmanto_cdh_2021}). By the $\mathbb{P}^1$-bundle formula for the presheaves $R\Gamma_{\text{ét}}(-,j_!\mu_{p^k}^{\otimes i})$ (\cite[proof of Theorem~$9.1.1$]{bhatt_absolute_2022}), there is a natural equivalence
		$$R\Gamma_{\text{ét}}(V,j_!\mu_{p^k}^{\otimes i}) \oplus R\Gamma_{\text{ét}}(V,j_!\mu_{p^k}^{\otimes (i-1)})[-2] \longrightarrow R\Gamma_{\text{ét}}(\mathbb{P}^1_V,j_!\mu_{p^k}^{\otimes i})$$
		in the derived category $\mathcal{D}(\Z/p^k)$. The functors $R\Gamma_{\text{ét}}(-,j_!\mu_{p^k}^{\otimes i})$ and $R\Gamma_{\text{ét}}(-,j_!\mu_{p^k}^{\otimes (i-1)})$ are moreover rigid (\cite{gabber_affine_1994}, see also \cite[Corollary~$1.18\,(1)$]{bhatt_arc-topology_2021}) and the valuation ring $V$ is $p$\nobreakdash-henselian, so the complex $R\Gamma_{\text{ét}}(\mathbb{P}^1_V,j_!\mu_{p^k}^{\otimes i}) \in \mathcal{D}(\Z/p^k)$ is zero. This implies that the complex $\big(L_{\text{cdh}} \tau^{>i} R\Gamma_{\text{ét}}(-,j_!\mu_{p^k}^{\otimes i})\big)(\mathbb{P}^1_V)$ is naturally identified with the complex $$\big(L_{\text{cdh}} \tau^{\leq i} R\Gamma_{\text{ét}}(-,j_!\mu_{p^k}^{\otimes i})\big)(\mathbb{P}^1_V)[1] \in \mathcal{D}(\Z/p^k),$$ and is thus in degrees at most~\hbox{$i+1$}.
	\end{proof}

	Following \cite{elmanto_cdh_2021}, we say that a $\mathcal{D}(\Z)$-valued presheaf on qcqs schemes satisfies {\it henselian $v$-excision} if for every henselian valuation ring $V$ and every prime ideal $\mathfrak{p}$ of $V$, this presheaf sends the bicartesian square of commutative rings
	$$\begin{tikzcd}
		V \ar[r] \ar[d] & V_{\mathfrak{p}} \ar[d] \\
		V/\mathfrak{p} \ar[r] & V_{\mathfrak{p}} / \mathfrak{p} V_{\mathfrak{p}}
	\end{tikzcd}$$
	to a cartesian square. Note that in the previous bicartesian square, all the commutative rings are henselian valuation rings by \cite[Lemma~$3.3.5$]{elmanto_cdh_2021}. The following lemma explains how to use henselian $v$-excision to prove that a map of cdh sheaves is an equivalence.
	
	\begin{lemma}\label{lemmaBhvexcisivefinitarycdhsheavesimplysufficientonrankatmost1henselianvaluationrings}
		Let $S$ be a qcqs scheme of finite valuative dimension, $\mathcal{C}$ be a $\infty$-category which is compactly generated by cotruncated objects, and $F,G : \emph{Sch}^{\emph{qcqs,op}}_S \rightarrow \mathcal{C}$ be finitary cdh sheaves satisfying henselian $v$-excision. Then a map of presheaves $F \rightarrow G$ is an equivalence of presheaves if and only if the map $F(V) \rightarrow G(V)$ is an equivalence in $\mathcal{C}$ for every henselian valuation ring $V$ of rank at most one with a map $\emph{Spec}(V) \rightarrow S$.
	\end{lemma}
	
	\begin{proof}
		By \cite[Proposition~$3.1.8\,(2)$]{elmanto_cdh_2021}, a map $F \rightarrow G$ is an equivalence of presheaves if and only if it is an equivalence on henselian valuation rings over $S$. The presheaves $F$ and $G$ being finitary, this is equivalent to the fact that it is an equivalence on henselian valuation rings of finite rank over $S$. By induction, and using henselian $v$-excision, this is in turn equivalent to the fact that it is an equivalence on henselian valuation rings of rank at most one over $S$.
	\end{proof}
	
	\begin{lemma}\label{lemmaAdescentresultforcdhhighdegreesyntomiccohomology}
		Let $p$ be a prime number. Then for any integers $i \geq 0$ and $k \geq 1$, the $\mathcal{D}(\Z/p^k)$-valued presheaves $\big(L_{\emph{cdh}} \tau^{>i} \Z/p^k(i)^{\emph{syn}}\big)(-)$ and $\big(L_{\emph{cdh}} \tau^{>i} \Z/p^k(i)^{\emph{syn}}\big)(\mathbb{P}^1_{-})$ are finitary cdh sheaves on qcqs schemes, and satisfy henselian $v$-excision.
	\end{lemma}
	
	\begin{proof}
		The presheaf $\Z/p^k(i)^{\text{syn}}$ is finitary, and the cdh sheafification of a finitary presheaf is a finitary cdh sheaf (\cite[Lemma~$4.57$]{bouis_motivic_2024}), so the presheaf $\big(L_{\text{cdh}} \tau^{>i} \Z/p^k(i)^{\text{syn}}\big)(-)$ is a finitary cdh sheaf. Covers in a site are stable under base change, so the presheaf $$\big(L_{\text{cdh}} \tau^{>i} \Z/p^k(i)^{\text{syn}}\big)(\mathbb{P}^1_{-})$$ is also a finitary cdh sheaf. Henselian valuation rings are local rings for the cdh topology, and the presheaves $\tau^{>i} R\Gamma_{\text{ét}}(-,j_!\mu_{p^k}^{\otimes i})$ and $\tau^{>i} \Z/p^k(i)^{\text{BMS}}$ are rigid (\cite{gabber_affine_1994} and \cite[Theorem~$2.27$]{bouis_motivic_2024}), so the presheaves $$\big(L_{\text{cdh}} \tau^{>i} R\Gamma_{\text{ét}}(-,j_!\mu_{p^k}^{\otimes i})\big)(-) \quad \text{and} \quad \big(L_{\text{cdh}} \tau^{>i} \Z/p^k(i)^{\text{BMS}}\big)(-)$$ satisfy henselian $v$-excision. By Lemma~\ref{lemmaGfibresequencecdhhighdegreesyntomiccohomologywithétaleandBMS} (applied for $j=i$ and after cdh sheafification), the presheaf $\big(L_{\text{cdh}} \tau^{>i} \Z/p^k(i)^{\text{syn}}\big)(-)$ then satisfies henselian $v$-excision. Finally, the presheaf $$\big(L_{\text{cdh}} \tau^{>i} \Z/p^k(i)^{\text{syn}}\big)(\mathbb{P}^1_{-})$$ satisfies henselian $v$\nobreakdash-excision, as a consequence of \cite[Lemma~$3.3.7$]{elmanto_cdh_2021}, and henselian $v$-excision for the presheaf $\big(L_{\text{cdh}} \tau^{>i} \Z/p^k(i)^{\text{syn}}\big)(-)$.
	\end{proof}

    \begin{lemma}\label{lemmaDcdhsheafificationofrigidisweaklyrigid}
		Let $B$ be a commutative ring, $\pi$ be an element of $B$, $\mathcal{C}$ be a presentable $\infty$-category, and $F : \emph{Alg}_{B} \rightarrow \mathcal{C}$ be a finitary and rigid functor. If the functor $F$ is zero on $B[\tfrac{1}{\pi}]$-algebas, then for every qcqs $B$-scheme $X$, the natural map
		$$\big(L_{\emph{cdh}} F\big)(X) \longrightarrow \big(L_{\emph{cdh}} F\big)(X_{B/\pi})$$
		is an equivalence in the $\infty$-category $\mathcal{C}$.
	\end{lemma}
	
	\begin{proof}
		Covers in a site are stable under base change, and the cdh sheafification of a finitary presheaf is finitary, so the presheaf $$(L_{\text{cdh}} F)(-_{B/\pi})$$ is a finitary cdh sheaf on qcqs $B$-schemes. It then suffices to prove that for every henselian valuation ring $V$ which is a $B$-algebra, the natural map
		$$F(V) \longrightarrow \big(L_{\text{cdh}} F\big)(V/\pi)$$
		is an equivalence in the $\infty$-category $\mathcal{C}$. The presheaf $L_{\text{cdh}} F$ is a finitary cdh sheaf, so it is invariant under nilpotent extensions. In particular, the natural map
		$$\big(L_{\text{cdh}} F\big)(V/\pi) \longrightarrow \big(L_{\text{cdh}} F\big)(V/\sqrt{(\pi)})$$
		is an equivalence in the $\infty$-category $\mathcal{C}$. The quotient of a henselian valuation ring by one of its prime ideals is a henselian valuation ring, so the target of the previous map is naturally identified with the object $F(V/\sqrt{(\pi)}) \in \mathcal{C}$. We finally prove that the natural map
		$$F(V) \longrightarrow F(V/\sqrt{(\pi)})$$
		is an equivalence in the $\infty$-category $\mathcal{C}$. If $\pi$ is invertible in the henselian valuation ring $V$, then both terms are zero by hypothesis on the functor $F$. And if $\pi$ is not invertible in $V$, then in particular the henselian local ring $V$ is $\pi$-henselian, and the result is a consequence of rigidity for the functor $F$.
	\end{proof}
	
	\begin{corollary}\label{corollaryEhighdegreecdhBMSsyntomiccohoisweaklyrigid}
		Let $p$ be a prime number, $B$ be a discrete valuation ring of mixed characteristic $(0,p)$, $\pi$ be a uniformizer of $B$, and $X$ be a qcqs $B$-scheme. Then for any integers $i \geq 0$ and $k \geq 1$, the natural map
		$$\big(L_{\emph{cdh}} \tau^{>i} \Z/p^k(i)^{\emph{BMS}}\big)(X) \longrightarrow \big(L_{\emph{cdh}} \tau^{>i} \Z/p^k(i)^{\emph{BMS}}\big)(X_{B/\pi})$$
		is an equivalence in the derived category $\mathcal{D}(\Z/p^k)$. 
	\end{corollary}
	
	\begin{proof}
		The functor 
		$$\tau^{>i} \Z/p^k(i)^{\text{BMS}} : \text{Alg}_{B} \longrightarrow \mathcal{D}(\Z/p^k)$$ 
		is finitary (\cite[Theorem~$2.25\,(2)$]{bouis_motivic_2024}) and rigid (\cite[Theorem~$2.27$]{bouis_motivic_2024}). The functor $\Z/p^k(i)^{\text{BMS}}$ is moreover zero on $\Z[\tfrac{1}{p}]$-algebras (\cite[Remark~$2.16$]{bouis_motivic_2024}), so the result is a consequence of Lemma~\ref{lemmaDcdhsheafificationofrigidisweaklyrigid}.
	\end{proof}
	
	\begin{remark}\label{remarkhighdegreecdhsheafifiedsyntomicisweaklyrigidgeneralised}
		One can prove similarly that Corollary~\ref{corollaryEhighdegreecdhBMSsyntomiccohoisweaklyrigid} holds for $B$ a general valuation ring of mixed characteristic $(0,p)$, where the base change to the characteristic $p$ field $B/\pi$ is replaced by the base change to the characteristic $p$ valuation ring $B/\sqrt{(p)}$.
	\end{remark}
	
	For every qcqs scheme $X$, the compatibilities between the motivic and syntomic first Chern classes of Section~\ref{subsectionfirstChernclasses} imply that the natural diagram
	$$\begin{tikzcd}[column sep = huge]
		\Z/p^k(i)^{\text{mot}}(X) \oplus \Z/p^k(i-1)^{\text{mot}}(X)[-2] \text{ } \ar[r,"\pi^\ast \oplus c_1^{\text{mot}}(\mathcal{O}(1))\pi^\ast"] \ar[d] & \Z/p^k(i)^{\text{mot}}(\mathbb{P}^1_X) \ar[d] \\
		\Z/p^k(i)^{\text{syn}}(X) \oplus \Z/p^k(i-1)^{\text{syn}}(X)[-2] \text{ } \ar[r,"\pi^\ast \oplus c_1^{\text{syn}}(\mathcal{O}(1))\pi^\ast"] & \Z/p^k(i)^{\text{syn}}(\mathbb{P}^1_X)
	\end{tikzcd}$$
	is commutative. We define the natural map
	$$\big(L_{\text{cdh}} \tau^{>i} \Z/p^k(i)^{\text{syn}}\big)(X) \oplus \big(L_{\text{cdh}} \tau^{>i-1} \Z/p^k(i-1)^{\text{syn}}\big)(X)[-2] \longrightarrow \big(L_{\text{cdh}} \tau^{>i} \Z/p^k(i)^{\text{syn}}\big)(\mathbb{P}^1_X)$$
	in the derived category $\mathcal{D}(\Z/p^k)$ as the map induced, via \cite[Theorem~$5.10$]{bouis_motivic_2024}, by taking cofibres along the vertical maps of this commutative diagram.
	
	\begin{proposition}\label{propositionCP1bundleformulaformixedcharvaluationringrank1}
		Let $p$ be a prime number, and $V$ be a rank one henselian valuation ring of mixed characteristic $(0,p)$. Then for any integers $i \geq 0$ and $k \geq 1$, the natural map
		$$\tau^{>i} \Z/p^k(i)^{\emph{syn}}(V) \oplus \big(\tau^{>i-1} \Z/p^k(i-1)^{\emph{syn}}(V)\big)[-2] \longrightarrow \big(L_{\emph{cdh}} \tau^{>i} \Z/p^k(i)^{\emph{syn}}\big)(\mathbb{P}^1_V)$$
		is an equivalence in the derived category $\mathcal{D}(\Z/p^k)$.
	\end{proposition}
	
	\begin{proof}
		The valuation ring $V$ is $p$-henselian, so the natural maps
		$$\tau^{>i} \Z/p^k(i)^{\text{syn}}(V)  \longrightarrow \tau^{>i} \Z/p^k(i)^{\text{BMS}}(V)$$
		and
		$$\big(\tau^{>i-1}\Z/p^k(i-1)^{\text{syn}}(V)\big)[-2] \longrightarrow \big(\tau^{>i-1} \Z/p^k(i)^{\text{BMS}}(V)\big)[-2]$$
		are equivalences in the derived category $\mathcal{D}(\Z/p^k)$. We first prove that the induced map
		$$\tau^{>i} \Z/p^k(i)^{\text{BMS}}(V) \oplus \big(\tau^{>i-1} \Z/p^k(i-1)^{\text{BMS}}(V)\big)[-2] \longrightarrow \big(L_{\text{cdh}} \tau^{>i} \Z/p^k(i)^{\text{BMS}}\big)(\mathbb{P}^1_V)$$
		is an equivalence in the derived category $\mathcal{D}(\Z/p^k)$. Let $\kappa$ be the residue field of $V$. By the rigidity property \cite[Theorem~$2.27$]{bouis_motivic_2024}, the natural maps
		$$\tau^{>i} \Z/p^k(i)^{\text{BMS}}(V) \longrightarrow \tau^{>i} \Z/p^k(i)^{\text{BMS}}(\kappa)$$
		and
		$$\big(\tau^{>i-1} \Z/p^k(i-1)^{\text{BMS}}(V)\big)[-2] \longrightarrow \big(\tau^{>i-1} \Z/p^k(i-1)^{\text{BMS}}(\kappa)\big)[-2]$$
		are equivalences in the derived category $\mathcal{D}(\Z/p^k)$. By Corollary~\ref{corollaryEhighdegreecdhBMSsyntomiccohoisweaklyrigid}, the natural map
		$$\big(L_{\text{cdh}} \tau^{>i} \Z/p^k(i)^{\text{BMS}}\big)(\mathbb{P}^1_V) \longrightarrow \big(L_{\text{cdh}} \tau^{>i} \Z/p^k(i)^{\text{BMS}}\big)(\mathbb{P}^1_{V/p})$$
		is an equivalence in the derived category $\mathcal{D}(\Z/p^k)$. The presheaf $(L_{\text{cdh}} \tau^{>i} \Z/p^k(i)^{\text{BMS}})(\mathbb{P}^1_{-})$ is moreover a finitary cdh sheaf, so it is invariant under nilpotent extensions. In particular, the natural map
		$$\big(L_{\text{cdh}} \tau^{>i} \Z/p^k(i)^{\text{BMS}}\big)(\mathbb{P}^1_{V/p}) \longrightarrow \big(L_{\text{cdh}} \tau^{>i} \Z/p^k(i)^{\text{BMS}}\big)(\mathbb{P}^1_{\kappa})$$
		is an equivalence in the derived category $\mathcal{D}(\Z/p^k)$. It then suffices to prove that the natural map
		$$\tau^{>i} \Z/p^k(i)^{\text{BMS}}(\kappa) \oplus \big(\tau^{>i-1} \Z/p^k(i-1)^{\text{BMS}}(\kappa)\big)[-2] \longrightarrow \big(L_{\text{cdh}} \tau^{>i} \Z/p^k(i)^{\text{BMS}}\big)(\mathbb{P}^1_{\kappa})$$
		is an equivalence in the derived category $\mathcal{D}(\Z/p^k)$, and this is a consequence of the $\mathbb{P}^1$-bundle formula on characteristic $p$ fields for the presheaves $\Z/p^k(i)^{\text{cdh}}$ (\cite[Section~$8$]{bachmann_A^1-invariant_2024}) and $L_{\text{cdh}} \Z/p^k(i)^{\text{BMS}}$ (\cite[Lemma~$5.17$]{elmanto_motivic_2023}).
		
		We prove now that the natural map
		$$\tau^{>i} \Z/p^k(i)^{\text{syn}}(V) \oplus \big(\tau^{>i-1} \Z/p^k(i-1)^{\text{syn}}(V)\big)[-2] \longrightarrow \big(L_{\text{cdh}} \tau^{>i} \Z/p^k(i)^{\text{syn}}\big)(\mathbb{P}^1_V)$$
		is an equivalence in the derived category $\mathcal{D}(\Z/p^k)$. By Lemma~\ref{lemmaGfibresequencecdhhighdegreesyntomiccohomologywithétaleandBMS} (applied for $j=i$ and after cdh sheafification), we just proved that the cofibre of this map is naturally identified with the complex
		$$\big(L_{\text{cdh}} \tau^{>i} R\Gamma_{\text{ét}}(-,j_!\mu_{p^k}^{\otimes i})\big)(\mathbb{P}^1_V) \in \mathcal{D}(\Z/p^k).$$
		By Example~\ref{exampleSelmerKtheory}, these complexes, indexed by integers $i \geq 0$, form the graded pieces of the filtered spectrum defined as the cofibre of the natural map of filtered spectra
		$$\text{Fil}^\star_{\text{mot}} \text{K}(X;\Z/p^k) \longrightarrow \text{Fil}^\star_{\text{mot}} \text{K}^{\text{Sel}}(X;\Z/p^k).$$
		The cofibre of the natural map $\text{K}(-;\Z/p^k) \rightarrow \text{K}^{\text{Sel}}(-;\Z/p^k)$, as a cofibre of two additive invariants of $\Z$-linear $\infty$-categories, is an additive invariant of $\Z$-linear $\infty$-categories and, as such, satisfies the $\mathbb{P}^1$-bundle formula (\cite[Lemma~$5.6$]{elmanto_motivic_2023}). This filtration then induces a spectral sequence
		$$E_2^{i,j} = \text{H}^{i-j}\big(\big(L_{\text{cdh}} \tau^{>-j} R\Gamma_{\text{ét}}(-,j_!\mu_{p^k}^{\otimes (-j)})\big)(\mathbb{P}^1_V)\big) \Longrightarrow 0.$$
		By Lemma~\ref{lemmaFweirdétalecohomologygroupisconcentratedincohomologicaldegreei+1}, for every integer $i \geq 0$, the complex $\big(L_{\text{cdh}} \tau^{>i} R\Gamma_{\text{ét}}(-,j_!\mu_{p^k}^{\otimes i})\big)(\mathbb{P}^1_V) \in \mathcal{D}(\Z/p^k)$ is concentrated in degree $i+1$, so this spectral sequence degenerates. This implies the desired equivalence.
	\end{proof}
	
	\begin{remark}\label{remarkP^1bundleformulaforfields}
		The proof of Proposition~\ref{propositionCP1bundleformulaformixedcharvaluationringrank1} uses a reduction to the case of fields of characteristic $p$, where the result is a consequence of the $\mathbb{P}^1$-bundle formula for the presheaves $\Z/p^k(i)^{\text{cdh}}$ (\cite[Section~$8$]{bachmann_A^1-invariant_2024}) and $L_{\text{cdh}} \Z/p^k(i)^{\text{BMS}}$ (\cite[Lemma~$5.17$]{elmanto_motivic_2023}). It is however possible to bypass these two results and prove directly the $\mathbb{P}^1$-bundle formula on fields of characteristic~$p$ for the presheaves $L_{\text{cdh}} \tau^{>i} \Z/p^k(i)^{\text{BMS}}$, by imitating the degeneration argument of \cite[Lemma~$5.17$]{elmanto_motivic_2023}.
	\end{remark}
	
	\begin{proposition}\label{propositionP^1bundleformulaforcdhlocalhighdegreesyntomiccohomology}
		Let $X$ be a qcqs scheme, and $p$ be a prime number. Then for any integers $i \geq 0$ and $k \geq 1$, the natural map
		$$\big(L_{\emph{cdh}} \tau^{>i} \Z/p^k(i)^{\emph{syn}}\big)(X) \oplus \big(L_{\emph{cdh}} \tau^{>i-1} \Z/p^k(i-1)^{\emph{syn}}\big)(X)[-2] \longrightarrow \big(L_{\emph{cdh}} \tau^{>i} \Z/p^k(i)^{\emph{syn}}\big)(\mathbb{P}^1_X)$$
		is an equivalence in the derived category $\mathcal{D}(\Z/p^k)$.
	\end{proposition}
	
	\begin{proof}
		The presheaves $L_{\text{cdh}} \tau^{>i} \Z/p^k(i)^{\text{syn}}$ and $\big(L_{\text{cdh}} \tau^{>i} \Z/p^k(i)^{\text{syn}}\big)(\mathbb{P}^1_{-})$ are finitary cdh shea\-ves on qcqs schemes, and satisfy henselian $v$-excision (Lemma~\ref{lemmaAdescentresultforcdhhighdegreesyntomiccohomology}). It then suffices to prove the desired equivalence for henselian valuation rings of rank at most one (Lemma~\ref{lemmaBhvexcisivefinitarycdhsheavesimplysufficientonrankatmost1henselianvaluationrings}). Let~$V$ be a henselian valuation ring of rank at most one. 
		
		If $p$ is invertible in the valuation ring $V$, then this is equivalent to proving that the natural map
		$$\tau^{>i} R\Gamma_{\text{ét}}(V,\mu_{p^k}^{\otimes i}) \oplus \big(\tau^{>i-1} R\Gamma_{\text{ét}}(V,\mu_{p^k}^{\otimes (i-1)})\big)[-2] \longrightarrow \big(L_{\text{cdh}} \tau^{>i} R\Gamma_{\text{ét}}(-,\mu_{p^k}^{\otimes i})\big)(\mathbb{P}^1_V)$$
		is an equivalence in the derived category $\mathcal{D}(\Z/p^k)$. For every integer $i \geq 0$, there is a fibre sequence of $\mathcal{D}(\Z/p^k)$-valued presheaves
		$$\Z/p^k(i)^{\text{cdh}}(-) \longrightarrow R\Gamma_{\text{ét}}(-,\mu_{p^k}^{\otimes i}) \longrightarrow L_{\text{cdh}} \tau^{>i} R\Gamma_{\text{ét}}(-,\mu_{p^k}^{\otimes i})$$
		on qcqs $\Z[\tfrac{1}{p}]$-schemes (\cite[Theorem~$7.14$]{bachmann_A^1-invariant_2024}). The desired equivalence is then a consequence of the $\mathbb{P}^1$\nobreakdash-bundle formula on qcqs $\Z[\tfrac{1}{p}]$-schemes for the presheaf $\Z/p^k(i)^{\text{cdh}}$ (\cite[Section~$8$]{bachmann_A^1-invariant_2024}) and the presheaf $R\Gamma_{\text{ét}}(-,\mu_{p^k}^{\otimes i})$ (\cite[proof of Theorem~$9.1.1$]{bhatt_absolute_2022}). 
		
		If $p$ is zero in the valuation ring $V$, then this is a consequence of the $\mathbb{P}^1$-bundle formula on qcqs $\F_p$-schemes for the presheaves $\Z/p^k(i)^{\text{cdh}}$ (\cite[Section~$8$]{bachmann_A^1-invariant_2024}) and $L_{\text{cdh}} \Z/p^k(i)^{\text{BMS}}$ (\cite[Theorem~$5.14$]{elmanto_motivic_2023}). 
		
		If $p$ is neither invertible nor zero in the valuation ring $V$, then $V$ is a rank one henselian valuation ring of mixed characteristic $(0,p)$, and the result is Proposition~\ref{propositionCP1bundleformulaformixedcharvaluationringrank1}.
	\end{proof}
	
	\begin{proof}[Proof of Theorem~\ref{theoremP^1bundleformulaformotiviccohomology}]
		It suffices to prove the result rationally, and modulo $p$ for every prime number~$p$. Rationally, the Atiyah--Hirzebruch spectral sequence degenerates (\cite[Theorem~$4.1$]{bouis_motivic_2024}), so the result is a consequence of the $\mathbb{P}^1$-bundle formula for algebraic $K$-theory (Section~\ref{subsectionfirstChernclasses}). Let $p$ be a prime number. By \cite[Theorem~$5.10$]{bouis_motivic_2024}, for every integer $i \in \Z$, there is a fibre sequence of $\mathcal{D}(\F_p)$-valued presheaves on qcqs schemes
		$$\F_p(i)^{\text{mot}}(-) \longrightarrow \F_p(i)^{\text{syn}}(-) \longrightarrow \big(L_{\text{cdh}} \tau^{>i} \F_p(i)^{\text{syn}}\big)(-).$$
		By \cite[Theorem~$9.1.1$]{bhatt_absolute_2022}, the natural map
		$$\pi^\ast \oplus c_1^{\text{syn}}(\mathcal{O}(1))\pi^\ast : \F_p(i)^{\text{syn}}(X) \oplus \F_p(i-1)^{\text{syn}}(X)[-2] \longrightarrow \F_p(i)^{\text{syn}}(\mathbb{P}^1_X)$$
		is an equivalence in the derived category $\mathcal{D}(\F_p)$. The result modulo $p$ is then a consequence of Proposition~\ref{propositionP^1bundleformulaforcdhlocalhighdegreesyntomiccohomology}.
	\end{proof}
	
	\subsection{Regular blowup and projective bundle formulae}
	
	\vspace{-\parindent}
	\hspace{\parindent}
	
	In this subsection, we prove the regular blowup formula for the motivic complexes $\Z(i)^{\text{mot}}$ (Theorem~\ref{theoremregularblowupformula}). By an argument of Annala--Iwasa, this and the $\mathbb{P}^1$-bundle formula imply the general projective bundle formula for the motivic complexes $\Z(i)^{\text{mot}}$ (Theorem~\ref{theoremprojectivebundleformula}).
	
	\begin{theorem}[Regular blowup formula]\label{theoremregularblowupformula}
		Let $Y \rightarrow Z$ be a regular closed immersion of qcqs schemes.\footnote{A morphism $Y \rightarrow Z$ is a regular closed immersion if it is a closed immersion, and if $Z$ admits an affine open cover such that $Y$ is defined by a regular sequence on each of the corresponding affine schemes.} Then for every integer $i \geq 0$, the commutative diagram
		$$\begin{tikzcd}
			\Z(i)^{\emph{mot}}(Z) \ar[r] \ar[d] & \Z(i)^{\emph{mot}}(Y) \ar[d] \\
			\Z(i)^{\emph{mot}}(\emph{Bl}_Y(Z)) \ar[r] & \Z(i)^{\emph{mot}}(\emph{Bl}_Y(Z) \times_Z Y)
		\end{tikzcd}$$
		is a cartesian square in the derived category $\mathcal{D}(\Z)$.
	\end{theorem}
	
	\begin{proof}
		It suffices to prove the result rationally, and modulo $p$ for every prime number $p$. By definition, a cdh sheaf sends an abstract blowup square to a cartesian square, and in particular satisfies the regular blowup formula. By \cite[Corollary~$4.67$]{bouis_motivic_2024}, the regular blowup formula for the presheaf $\Q(i)^{\text{mot}}$ is then equivalent to the regular blowup formula for the presheaf $R\Gamma_{\text{Zar}}(-,\mathbb{L}\Omega^{<i}_{-_{\Q}/\Q})$. And the regular blowup formula for the presheaf $R\Gamma_{\text{Zar}}(-,\mathbb{L}\Omega^{<i}_{-_{\Q}/\Q})$ is a consequence of the fact that for every integer $j \geq 0$, the presheaf $R\Gamma_{\text{Zar}}(-,\mathbb{L}^j_{-_/\Z} \otimes_{\Z} \Q)$ satisfies the regular blowup formula (\cite[Lemma~$9.4.3$]{bhatt_absolute_2022}). 
		
		Let $p$ be a prime number. Similarly, \cite[Corollary~$3.26$]{bouis_motivic_2024} implies that the regular blowup formula for the presheaf $\F_p(i)^{\text{mot}}$ is equivalent to the regular blowup formula for the presheaf $\F_p(i)^{\text{BMS}}$. By \cite[Corollary~$5.31$]{antieau_beilinson_2020}, there exists an integer $m \geq 0$ and an equivalence of presheaves
		$$\F_p(i)^{\text{BMS}}(-) \xlongrightarrow{\sim} \text{fib} \Big(\text{can}-\phi_i : (\mathcal{N}^{\geq i} \Prism_{-}\{i\}/\mathcal{N}^{\geq i+m} \Prism_{-}\{i\})/p \longrightarrow (\Prism_{-}\{i\}/\mathcal{N}^{\geq i+m} \Prism_{-}\{i\})/p\Big).$$
		In particular, it suffices to prove that for every integer $j \geq 0$, the presheaf $\mathcal{N}^j \Prism_{-}/p$ satisfies the regular blowup formula. By \cite[Remark~$5.5.8$ and Example~$4.7.8$]{bhatt_absolute_2022}, there is a fibre sequence of presheaves
		$$\mathcal{N}^j \Prism_{-} \{i\}/p \longrightarrow \text{Fil}^{\text{conj}}_j \overline{\Prism}_{-/\Z_p\llbracket \widetilde{p} \rrbracket}/p \xlongrightarrow{\Theta + j} \text{Fil}^{\text{conj}}_{j-1} \overline{\Prism}_{-/\Z_p\llbracket \widetilde{p} \rrbracket}/p.$$
		The presheaves $\text{Fil}^{\text{conj}}_j \overline{\Prism}_{-/\Z_p\llbracket \widetilde{p} \rrbracket}/p$ and $\text{Fil}^{\text{conj}}_{j-1} \overline{\Prism}_{-/\Z_p\llbracket \widetilde{p} \rrbracket}/p$ have finite filtrations with graded pieces given by modulo $p$ powers of the cotangent complex, and the result is then a consequence of the regular blowup formula for powers of the cotangent complex (\cite[Lemma~9.4.3]{bhatt_absolute_2022}).
	\end{proof}
	
	\begin{theorem}[Projective bundle formula]\label{theoremprojectivebundleformula}
		Let $X$ be a qcqs scheme, $r \geq 1$ be an integer, $\mathcal{E}$~be a vector bundle of rank $r+1$ on $X$, and $\pi : \mathbb{P}_X(\mathcal{E}) \rightarrow X$ be the projectivisation of $\mathcal{E}$. Then for every integer $i \in \Z$, the natural map
		$$\sum_{j=0}^r c_1^{\emph{mot}}(\mathcal{O}(1))^j \pi^{\ast} : \bigoplus_{j=0}^r \Z(i-j)^{\emph{mot}}(X)[-2j] \longrightarrow \Z(i)^{\emph{mot}}(\mathbb{P}_X(\mathcal{E}))$$
		is an equivalence in the derived category $\mathcal{D}(\Z)$.
	\end{theorem}
	
	\begin{proof}
		By Zariski descent, it suffices to consider the case where the vector bundle $\mathcal{E}$ is given by $\mathbb{A}^{r+1}_X$, {\it i.e.}, to prove that the natural map 
		$$\sum_{j=0}^r c_1^{\text{mot}}(\mathcal{O}(1))^j \pi^{\ast} : \bigoplus_{j=0}^r \Z(i-j)^{\text{mot}}(X)[-2j] \longrightarrow \Z(i)^{\text{mot}}(\mathbb{P}^r_X)$$
		is an equivalence in the derived category $\mathcal{D}(\Z)$. The presheaves $\Z(i)^{\text{mot}}$ satisfy the $\mathbb{P}^1$-bundle formula (Theorem~\ref{theoremP^1bundleformulaformotiviccohomology}). Moreover, for every qcqs scheme $X$ and every integer $m \geq 0$, they send the blowup square
		$$\begin{tikzcd}
			\mathbb{P}^m_X \ar[r] \ar[d] & \text{Bl}_X(\mathbb{A}^{m+1}_X) \ar[d] \\
			X \arrow[r,"0"] & \mathbb{A}^{m+1}_X
		\end{tikzcd}$$
		to a cartesian square in the derived category $\mathcal{D}(\Z)$ (Theorem~\ref{theoremregularblowupformula}, in the special case where the regular closed immersion $Y \rightarrow Z$ is the zero section $X \rightarrow \mathbb{A}^{m+1}_X$).
		By the argument of \cite[Lemma~$3.3.5$]{annala_motivic_2023}, these two properties imply, by induction, the desired projective bundle formula.
	\end{proof}
	
	In the following result, denote by $\Z(i)^{\text{mot}}_X : \text{Sm}^{\text{op}}_X \rightarrow \mathcal{D}(\Z)$ the Zariski sheaves on smooth schemes over $X$ induced by restriction of the motivic complexes $\Z(i)^{\text{mot}}$.
	
	\begin{corollary}[Motivic cohomology is represented in motivic spectra]\label{corollaryP1motivicspectra}
		For every qcqs scheme $X$, the motivic complexes $\{\Z(i)^{\emph{mot}}_X\}_{i \in \Z}$ are represented by a $\mathbb{P}^1$-motivic spectrum in the sense of \cite{annala_motivic_2023}.
	\end{corollary}
	
	\begin{proof}
		By definition of $\mathbb{P}^1$-motivic spectra, this is a consequence of elementary blowup excision (which is a special case of Theorem~\ref{theoremregularblowupformula}) and the $\mathbb{P}^1$-bundle formula (Theorem~\ref{theoremP^1bundleformulaformotiviccohomology}).
	\end{proof}

	
	
	
	
	\bibliographystyle{alpha}
	
	{\footnotesize
\bibliography{biblio.bib}
}

\end{document}